\documentclass{amsart}
\usepackage{amsmath, amssymb, amsthm, latexsym, enumerate, mathrsfs, textcomp, bbm, caption, url, overpic, appendix}

\usepackage[colorlinks=true, pdfstartview=FitV, linkcolor=blue, citecolor=blue, urlcolor=blue]{hyperref}

\usepackage{ifthen}
\usepackage[T1]{fontenc}

\newcommand{\showcomments}{yes}

\newsavebox{\commentbox}
{\ifthenelse{\equal{\showcomments}{yes}}%
{\footnotemark
    \begin{lrbox}{\commentbox}
    \begin{minipage}[t]{1.25in}\raggedright\sffamily\upshape\tiny
    \footnotemark[\arabic{footnote}]}
{\begin{lrbox}{\commentbox}}}%
{\ifthenelse{\equal{\showcomments}{yes}}%
{\end{minipage}\end{lrbox}\marginpar{\usebox{\commentbox}}}
{\end{lrbox}}}

\theoremstyle{plain}
\newtheorem{theorem}{Theorem}[section]
\newtheorem{corollary}[theorem]{Corollary}
\newtheorem{lemma}[theorem]{Lemma}
\newtheorem{proposition}[theorem]{Proposition}

\newtheorem{conjecture}[theorem]{Conjecture}

\theoremstyle{definition}
\newtheorem{example}{Example}[section]

\numberwithin{figure}{section}

\DeclareMathOperator{\CAT}{CAT}
\DeclareMathOperator{\vcd}{vcd}
\DeclareMathOperator{\Lk}{Lk}

\newcommand{\Swiatkowski}{{\'{S}}wi{\k{a}}tkowski}

\newcommand{\K}{K_{3,3}}
\newcommand{\G}{\Gamma}
\newcommand{\B}{\Pi}
\newcommand{\bndry}{\partial}
\newcommand{\D}{\mathcal{D}}
\newcommand{\bbh}{\mathbb{H}}

\renewcommand{\L}{\Lambda}

\title{Right-angled Coxeter groups with non-planar boundary}

\begin{document} 
\author[P. Dani]{Pallavi Dani} 

\author[M. Haulmark]{Matthew Haulmark} 

\author[G. S. Walsh]{Genevieve Walsh}

\thanks{This work of the first author was supported
by a grant from the Simons Foundation (\#426932, Pallavi Dani) and by NSF Grant No.~DMS-1812061. The third author was supported by NSF Grant No. ~DMS 1709964. Some of this work was conducted at the Centre International de Rencontres Math\'ematiques in Luminy and we thank them for their hospitality. We would like to thank Kevin Schreve for his comments on a draft of this paper.   } 

\begin{abstract} We investigate the planarity of the boundaries of right-angled Coxeter groups. 
We show that non-planarity of the defining graph
 does not necessarily imply non-planarity of every boundary of the associated right-angled Coxeter group, although it does in many cases. 
Our techniques yield a characterization of the triangle-free defining graphs such that the associated right-angled Coxeter group has boundary a Menger curve. \end{abstract}

\maketitle

\section{Introduction}

Here we investigate boundaries of certain $\CAT(0)$ groups.   A group is $\CAT(0)$ if it acts geometrically (properly discontinuously, co-compactly and by isometries) on a $\CAT(0)$ space.  Every $\CAT(0)$ metric space $X$ has a well-defined visual boundary~$\partial X$.   We will denote a proper $\CAT(0)$ space on which $G$ acts geometrically by~$X_G$, or by $X_\G$, when $G$ is the right-angled Coxeter group defined by a graph $\G$. See Section~\ref{sec:prelim} for more detailed  definitions. 

When a group $G$ acts geometrically on $X_G$, the topology of $\bndry X_G$ can provide information about the algebra of $G$, even though the boundary of $G$ may not be well-defined.  The dimension of $\bndry X_G$ is closely related to the cohomological dimension of 
$G$~\cite{GO07, BM}, and in the case that $G$ is torsion-free, $G$ is a $\text{PD}(3)$ group exactly when~$\partial X_G \cong S^2$~\cite{BM}.  Also, splittings of $G$ are expressed as topological features in~$\bndry X_G$ (see~\cite{Bow98, PS09, Hau18b} among others). 

An important question about the topology of boundaries is their planarity. We say that a topological space is {\it planar} if it can be embedded in $S^2$. When $G$ can be virtually realized as a geometrically finite Kleinian group, every $\CAT(0)$ boundary~$\partial X_G$ is planar.  This can be seen as follows.  The limit set of the Kleinian group is a subset of $S^2$.  Moreover,  the $\CAT(0)$ boundary is well-defined, since $G$ is either hyperbolic or $\CAT(0)$ with isolated flats \cite{HK1}. This boundary is either the limit set itself, or the limit set with parabolic fixed points replaced by circles.  In either case, it is planar.  By a special case of a theorem of 
Bestvina--Kapovich--Kleiner~\cite{BKK02}, if $G$ is the fundamental group of a 3-manifold, then 
no boundary of $G$ contains a $K_5$ or a $K_{3,3}$ (though this doesn't immediately imply that every boundary is planar, as we discuss below).

Conjecture~\ref{conj: the conjecture} below, which was asked as questions in~\cite[Questions 1.3 and 1.4]{SchSta18}, 
presents a sort of converse to the above statements.  This paper contains some evidence for 
Conjecture~\ref{conj: the conjecture} in the setting of right-angled Coxeter groups.   See Corollary~\ref{cor:planar}.

\begin{conjecture}\label{conj: the conjecture} Let $G$ be a $\CAT(0)$ group with a planar visual boundary.  Then every visual boundary of $G$ is planar, and furthermore,  $G$ is virtually the fundamental group of a compact 3-manifold. \end{conjecture}

By work of Ha\"{i}ssinsky~\cite[Theorem 1.10]{Hai15}, this conjecture is known to hold for hyperbolic groups which are $\CAT(0)$ cubed, and hence for hyperbolic right-angled Coxeter groups.   We note that Conjecture~\ref{conj: the conjecture} implies the Cannon Conjecture~\cite{Can89} for hyperbolic groups which are $\CAT(0)$.  As we write there are no known examples of hyperbolic groups which are not $\CAT(0)$. 

 Conjecture~\ref{conj: the conjecture} is more speculative than the analogous conjecture for hyperbolic groups in ~\cite[Conjecture 1.6]{Hai15}. This is because $\CAT(0)$ boundaries, unlike the boundaries of hyperbolic groups, are not always well-defined \cite{CK00} and not always locally connected~\cite{MR99}. This last property means that a result of 
 Claytor~\cite[Theorem~C]{Claytor} does not necessarily apply. In particular, when $\bndry X$ is not locally connected non-planarity of the boundary does not imply that there is a $K_{3,3}$ or $K_5$ in the boundary.   
 Schreve and Stark~\cite{SchSta18} have an example of two homeomorphic $\CAT(0)$ complexes with two different boundaries, one of which contains an embedded $K_{3,3}$ and one of which does not.  Both boundaries are non-planar.

In this paper we study the planarity of boundaries of right-angled Coxeter groups. 
Given a finite simplicial graph $\Gamma$ we denote the 
associated right-angled Coxeter group by $W_{\Gamma}$.  Every $W_\Gamma$ is $\CAT(0)$, and in particular, $W_\Gamma$ acts geometrically on a $\CAT(0)$ cube complex $\Sigma_{\Gamma}$ called its {\it Davis--Moussong complex}.

It is tempting to conjecture that if $\Gamma$ is non-planar then every $\CAT(0)$ boundary of $W_\G$ is non-planar, up to a finite subgroup.  Indeed, {\Swiatkowski } speculates that planarity of the defining graph may be a necessary condition for a planar boundary.  (See~\cite[Remark 3]{Swi16}.) 
However, we prove that non-planarity of the defining graph of a right-angled Coxeter group does not guarantee non-planarity of the boundary; there is a $W_\Gamma$ with a non-planar defining graph $\G$, which has a planar $\CAT(0)$ boundary, as we explain in Example~\ref{ex:bad-graph} below. (Indeed there are many.)  
\begin{figure}[h!] 
\begin{center}
\begin{overpic}[scale=0.8]{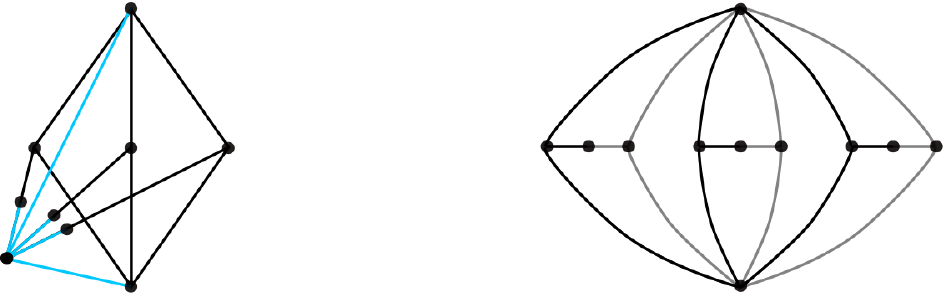}
\put(13,32){\scriptsize $x$}
\put(0,15){\scriptsize $a$}
\put(-3, 2){\scriptsize $y$}
\put(15,15){\scriptsize $b$}
\put(13,-2){\scriptsize $z$}
\put(25.5,15){\scriptsize $c$}
\put(78,32){\scriptsize $x$}
\put(55,15){\scriptsize $a$}
\put(67.5,15){\scriptsize $a'$}
\put(71.5,15){\scriptsize $b$}
\put(84,15){\scriptsize $b'$}
\put(87,15){\scriptsize $c$}
\put(100,15){\scriptsize $c'$}
\put(78,-2){\scriptsize $z$}
\put(-2,28){\scriptsize $\L$}
\put(96,28){\scriptsize $\L'$}
\end{overpic}
\end{center}
\caption {In the graph $\L$ on the left, each blue segment is an edge, while each black segment is a path which may or may not be subdivided.  The graph $\L'$ on the right is the double of $\L$ over the vertex $y$, as defined in Section~\ref{sec:RACG}.  The two copies  of $\L$ minus the open star of $y$ are shown in $\L'$ in black and grey respectively.
}
\label{fig-badgraph-double}
\end{figure}
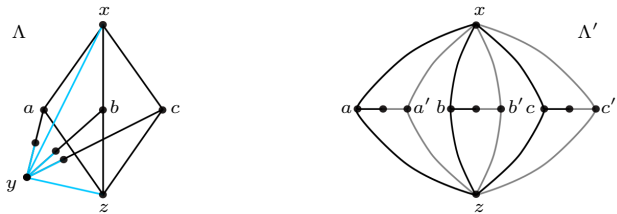

\begin{example}\label{ex:bad-graph}
Let $\L$ denote the graph on the left in Figure~\ref{fig-badgraph-double}, which is non-planar.  
As we observe in Lemma~\ref{lem:doubling}, the group $W_\L$ contains an index two subgroup isomorphic to $W_{\L'}$, 
where $\L'$ is the planar graph on the right of 
Figure~\ref{fig-badgraph-double}.  Any right-angled Coxeter group with planar defining graph is virtually the fundamental group of a 3-manifold.  Indeed, the defining graph can be embedded as an induced subgraph in a triangulation $T$ of a 2-sphere.  
Now $W_T$, 
the right-angled Coxeter group defined by the one-skeleton of $T$, is virtually a closed 3-manifold group, since the Davis--Moussong complex $\Sigma_T$  of $W_T$  is a manifold.  The 
Davis--Moussong complex $\Sigma_{\L'}$ of~$W_{\L'}$ is a convex subcomplex of $\Sigma_T$.  Consequently $\partial \Sigma_{\L'}$ embeds in $\partial \Sigma_T \cong  S^2$, 
so $\partial \Sigma_{\L'}$ is planar.  (See also~\cite{DO01}.)   
Now $W_\L$ acts on $W_{\L'}$ by  conjugation, and this induces a geometric action of $W_\L$ on 
on 
$\Sigma_{\L'}$.
It follows that $W_\L$ has a planar boundary as well.   
\end{example}

The graph $\L$ in Figure~\ref{fig-badgraph-double} is a subdivision of the graph in Figure~\ref{fig-badgraph}.  More generally, 
let $\B$ denote 
any  
graph obtained by subdividing the black segments 
of the graph in Figure~\ref{fig-badgraph}
enough to get a triangle-free graph.  Then $W_\B$ has a planar boundary by a similar argument.  
However, we show that, in a certain sense, the graph $\B$ is the only obstruction:

\begin{theorem}\label{thm:main1}
Let $\G$ be a triangle-free non-planar graph, and let $X_{\G}$ be a proper $\CAT(0)$ space on which $W_\G$ acts geometrically.   Then either $\partial X_{\G}$ is non-planar or $W_\G$ contains a finite-index special subgroup 
whose defining graph contains an induced copy of the graph $\B$ in Figure~\ref{fig-badgraph}.  
\end{theorem}

\begin{figure}[h!]
\begin{center}
\begin{overpic}[scale=0.7]{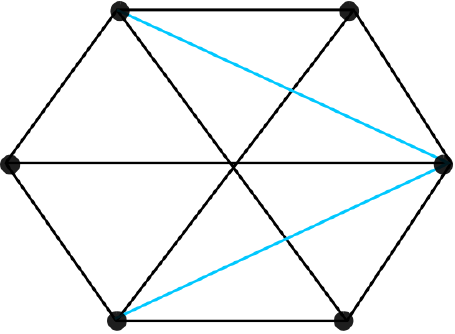}
\end{overpic}
\end{center}
\caption {The graph $\B$ is defined to be any graph as shown in this figure, such that the 
blue segments are edges, the black segments may or may not be edges, and enough 
 of the black segments are subdivided to ensure that the result is triangle-free. A specific instance of such a subdivision appears on the left in Figure~\ref{fig-badgraph-double}.
}
\label{fig-badgraph}

\end{figure}

A graph is \emph{inseparable} if it is connected, has no separating complete subgraph, no cut pair, and no separating complete subgraph suspension. When $\G$ is inseparable and $\partial X_\G$ is locally connected and planar,  it follows from our Corollary~\ref{cor:badnoplanar}, that $\partial X_\G$  cannot contain an induced copy of the graph in Figure~\ref{fig-badgraph}.  Thus we have the following:

\begin{theorem}\label{thm:main2}
Let $\G$ be a triangle-free inseparable graph and let $X_\Gamma$ be a $\CAT(0)$ space on which $W_{\G}$ acts geometrically.  If $\G$ is non-planar and $\bndry X_\G$ is locally connected and contains no local cut points, then $\bndry X_\G$ is non-planar.  
\end{theorem}

 Using Theorem~\ref{thm:main2} we can now conclude that Conjecture~\ref{conj: the conjecture} holds for a class of right-angled Coxeter groups:

\begin{corollary} (to Theorem~\ref{thm:main2})
\label{cor:planar} Let $\G$ be a graph with no triangles, and $X_\G$ a proper $\CAT(0)$ space on which $W_\G$ acts geometrically.   Then if $\partial X_G$ is a Sierpinski carpet, $W_{\G}$ is virtually a 3-manifold group..  
\end{corollary}

The corollary holds as follows.  Since $\partial X$ is a Sierpinski carpet, it is planar, locally connected, and has no local cut points or cut pairs.  In the proof of Theorem~\ref{thm:main2}, 
the hypotheses are used to conclude that $\partial X$, if planar, it is a Sierpinski carpet.  Therefore in 
this setting, we can conclude from Theorem~\ref{thm:main2},  that the defining graph  $\G$ is planar as well. 
 This implies that $W_\G$ is virtually a 3-manifold group by the argument in Example~\ref{ex:bad-graph}.

Theorem~\ref{thm:main2} also allows us to characterize the right-angled Coxeter groups with 
triangle-free defining graphs which have Menger curve boundaries.  

\begin{corollary} (to Theorem~\ref{thm:main2}) 
\label{cor:Menger}
Let $\G$ be a triangle-free inseparable graph, such that $W_{\G}$ is either hyperbolic or $\CAT(0)$ with isolated flats.  Then the following are equivalent:
\begin{enumerate} 
\item $\G$ is non-planar. 
\item Every $\CAT(0)$ boundary $\partial X_\G$ is Menger curve. 
\item Some $\CAT(0)$ boundary $\partial X_\G$ is a Menger curve. 
\end{enumerate} 
\end{corollary}

\begin{proof}
We claim that under the hypotheses of the corollary, $\partial X_\G$ is 1-dimensional for every 
CAT(0) space $X_\G$ on which $W_\G$ acts geometrically.  
Since $\G$ is triangle-free, the Davis--Moussong complex $\Sigma_\G$ is $2$-dimensional, and the  virtual cohomological  dimension ($\vcd$) of $W_\G$ is equal to $1$ or $2$. By Stallings~\cite{Sta68}  a group $G$ with $\vcd(G)=1$ is virtually free.  Since $\G$ is inseparable, the group $W_\G$ is not virtually free, so $\vcd(W_\G)=2$.  
Since Coxeter groups are virtually torsion-free, a theorem of Bestvina-Mess~\cite[Corollary 1.4]{BM}, see also~\cite[Theorem~1.7]{Bes96},
 implies that the covering dimension of $\bndry X_\G$ is equal to $\vcd(W_\G)-1$. 
 This proves the claim. 
 
 Now suppose $\G$ is non-planar. 
 When $W_\G$ is hyperbolic, $\partial X_\G$ is locally connected (see~\cite{BM, Swa, Bow99a}). 
 Since 
$\vcd(W_\G)=2$, the highest rank of a virtually abelian subgroup is 2, and we apply \cite{HR1} to 
 conclude that $\partial X_\G$ is locally connected in the $\CAT(0)$ with isolated flats case.  
 Since $\G$ is inseparable, $W_\G$ does not split over a 2-ended subgroup.  Thus 
 by~\cite[Theorem~6.2]{Bow98} (in the hyperbolic case) or by~\cite[Theorem~1.3]{Hau18b} (in the isolated flats case) we conclude that $\partial X_\G$ has no local cut points.  
 Now we may apply  Theorem~\ref{thm:main2} to conclude 
 that $\partial X_\G$ is non-planar.  Then by~\cite[Theorem~4]{KK00} (in the hyperbolic case)
or by~\cite[Theorem 1.2]{Hau18b} (in the isolated flats case), we conclude that $\partial X_\G$ is a Menger curve.  
 
On the other hand if $\G$ is planar, then $\partial \Sigma_\G$ is planar (as we showed in 
Example~\ref{ex:bad-graph}), and cannot be a Menger curve.  Conditions (2) and  (3) are equivalent because in this situation, the $\CAT(0)$ boundary is well defined \cite{HK1}. 
\end{proof}

We remark that~\cite[Theorem 1]{Swi16} implies (in the setting of triangle-free graphs) that if $\G$ is inseparable and $W_\G$ is hyperbolic, then the Gromov boundary $\partial W_\G$ is a Sierpinski carpet.   {\Swiatkowski } suggests in Remark 3 of~\cite{Swi16} that planarity of the nerve may be a necessary condition for a Coxeter group to have Sierpinski carpet boundary (up to a product with a finite Coxeter group).  Corollary~\ref{cor:Menger} shows that this is true in the case of hyperbolic right-angled Coxeter groups  defined by triangle-free graphs.

 In the setting of hyperbolic groups, Menger curve boundary is known to be generic~\cite{DGP11}.  Recently Haulmark--Hruska--Sathaye~\cite{HHS18} provide examples of large type (i.e.~not right-angled) Coxeter groups which are not hyperbolic and have visual boundary homeomorphic to the Menger curve.  Corollary~\ref{cor:Menger} provides a technique for constructing large classes of examples of right-angled Coxeter groups with Menger curve boundary.  The following example gives a concrete such class.

\begin{example}
One class of finite simplicial graphs which yield right-angled Coxeter groups with Menger curve boundary are the Mobius Ladders (see \url{https://en.wikipedia.org/wiki/Mobius\_ladder}) 
These graphs are inseparable and non-planar. By a result of Caprace~\cite{Cap09, Cap15} the right-angled Coxeter groups defined by these graphs have isolated flats; therefore, they satisfy the hypotheses of Corollary~\ref{cor:Menger} and have Menger curve boundary.

\end{example}

\subsection{Overview of the paper}  In Section~\ref{sec:prelim}, we give some preliminaries on right-
angled Coxeter groups and their boundaries.  In Section~\ref{sec:graph} we use graph-theoretic 
techniques to show that if there is not an induced $K_{3,3}$ subdivision in our non-planar graph $\G$, 
then by taking the double over some vertex finitely many times, we arrive at a graph $\G'$ such that 
$W_{\G'}$ is a finite-index subgroup of $W_\G$, and~$\G'$ either contains an induced $K_{3,3}$ 
subdivision or an induced subdivided copy of one of two specific graphs. (See Figure~\ref{fig:lambdas}.) One of the two specific graphs is $\B$ from Figure~\ref{fig-badgraph}.
In 
Section~\ref{sec:boundariesnonplanar} we show 
if the defining graph of a right-angled Coxeter group contains either 
 an induced $K_{3,3}$ subdivision, or an 
induced copy of one of the two specific graphs (the one not equal to $\B$), then any boundary  $\partial X_\G$ is non-planar.  
Finally in 
Section~\ref{sec:bad} we deal with the case of the remaining graph $\B$.  We show that  if any visual 
boundary of a right-angled Coxeter group is connected, locally connected, has no local cut points and 
is planar, then the defining graph of that right-angled Coxeter group cannot contain a copy of $\B$. 
Theorem~\ref{thm:main1} is proven in Section~\ref{sec:boundariesnonplanar}, while 
Theorem~\ref{thm:main2} is proven in Section~\ref{sec:bad}.

%%%%%%%%%%%%%%%%%%%%%%%%%%%%%%%%%%%%%%%%%%%%%%%%%%%%%%%%%%%%%%
\section{Preliminaries} \label{sec:prelim} 

%%%%%%%%%%%%%%%%%%%%%%%%%%%%%%%%%%%%%%%%%%%%%%%%%%%%%%%%%%%%%%

\subsection{Boundaries of $\CAT(0)$ Spaces}
\label{subsec:CAT0Boundary}
%%%%%%%%%%%%%%%%%%%%%%%%%%%%%%%%%%%%%%%%%%%%%%%%%%%%%%%%%%%%%%%%%%%%%%%%
%
Let $X$ be a proper $\CAT(0)$ space.
The \emph{visual} or {\it $\CAT(0)$ boundary} of $X$, denoted $\bndry X$, is the set of equivalence classes of geodesic rays, where two rays $c_1,c_2\colon [0,\infty)\to X$ are equivalent if there exists a constant $D\geq 0$ such that $d\big(c_1(t),c_2(t)\big)\leq D$ for all $t\in[0,\infty)$.

The boundary $\bndry {X}$ comes equipped with a natural topology called the {\it cone topology}.
To define this topology, identify $\partial X$ with the set of geodesic rays based at some fixed point $p $ in $X$.  
Then if $c$ is a geodesic ray based at $p$, a basic open set around $c$ consists of 
geodesic rays based at $p$ whose projection onto a ball of radius $t$ around~$p$ is close to $c(t)$.

If $G$ acts geometrically on $X$ one would like to define $\bndry G$ to be $\bndry X$. If $G$ is a hyperbolic group, then $X$ is a Gromov hyperbolic metric space and $\partial X$ is the Gromov boundary of $G$.  In particular, in this case the $\CAT(0)$ boundary $\partial G$ is well-defined.  For example, if $G$ is virtually free, then the boundary of any $\CAT(0)$ space that $G$ acts on geometrically is a Cantor set.  Hruska--Kleiner have shown that $\bndry G$ is also well-defined in the setting of $\CAT(0)$ groups with isolated flats~\cite{HK1}. In general the homeomorphism type of the boundary  is not well-defined for $\CAT(0)$ groups (see~\cite{CK00, SchSta18}). However, for special subgroups of Coxeter groups, one can find a boundary for that special subgroup in any $\CAT(0)$ boundary for the Coxeter group, by the following lemma of Mihalik and Tschantz \cite{MT13}. Suppose that $(W, S)$ is a finitely generated Coxeter system, $\mathcal C$ is the Cayley graph of $W$ with
respect to $S$, and $W$ acts geometrically on a $\CAT(0)$ space $X$. Fix a point $x \in X$, and define a graph  $\mathcal C_x \subset X $ to have as vertices the orbit $W\cdot x$ and as edges the collection of $\CAT(0)$ geodesic paths connecting $wx$ and $wsx$, for $w \in W$ and $s \in S$.  Note that the collection of edges will be $W$-equivariant by uniqueness of $\CAT(0)$ geodesics. 

\begin{lemma}[\cite{MT13} Corollary 6.6] \label{lem:qc!} Suppose that $(W,S)$ is a finitely generated Coxeter group with Cayley graph $\mathcal C$, acting geometrically on the $\CAT(0)$ space $X$, and take an $x \in X$, and $P_x : \mathcal C \rightarrow \mathcal C_x$ mapping $\mathcal C$ quasi-isometrically and $W$-equivariantly into $X$. Then for each subset $A \subset S$,  (the image of) the subgroup  $<A>$ is quasi-convex in $X$. 
\end{lemma}

 We will use the lemma below often to find circles and Cantor sets in the boundary $\partial X$ of some $\CAT(0)$ space for $W_\Gamma$. We will also use it to refer to points of the boundary.  For example, if $x$ and $y$ are disjoint vertices on $\Gamma$, the special subgroup defined by $x$ and $y$ is virtually cyclic and we will refer to the points of its boundary as $(xy)^\infty$ and $(yx)^\infty$. 

 \begin{lemma} \label{lem:inboundary} 

Let $\Gamma$ be a graph and $W_\Gamma$ the right-angled Coxeter group defined by~$\Gamma$. Suppose that $A$ is a subset of the vertices of $\Gamma$.  Let $X$ be any $\CAT(0)$ space with a geometric action of $W_\Gamma$.  Let $W_A$ be the special subgroup defined by $A$. Then there is a naturally embedded copy of a boundary of $W_A$ in $\partial X_\Gamma$. In particular, if $A$ induces a cycle in $\Gamma$ then there is a circle in $\partial X$. 
\end{lemma} 

\begin{proof}  Lemma \ref{lem:qc!} above implies that the orbit $W_A \cdot x$ of a point $x \in X$ under the action of $W_A$  is quasi-convex.  The convex hull  of $W_A\cdot x$ in $X$ is the union of all geodesics in $X$ connecting points of $W_A\cdot x$.  Call this convex hull $C(W_A)$. Then $W_A$ acts geometrically on $C(W_A)$ and $C(W_A)$ is $\CAT(0)$.   Consider the $\CAT(0)$ boundary of $X$ as rays from the basepoint $x$.  Then  $\partial C(W_A)$ is naturally a subset of $\partial X$, and $\partial C(W_A)$ is a boundary of $W_A$.  In the case that $A$ induces a cycle, the boundary of $W_A$ is well-defined \cite{HK1} and is a circle. 
\end{proof}

\subsection{Right-angled Coxeter groups}
\label{sec:RACG}

Let $\Gamma$ be a finite simplicial graph. The {\it right-angled Coxeter group}  associated to $\G$ has generating set $S$ equal to the vertices of $\Gamma$, relations $s^2=1$ for each $s$ in $S$ and relations $st = ts$ whenever $s$ and $t$ are adjacent vertices of $\G$. Given a graph $\Gamma$ we denote the associated right-angled Coxeter group by $W_\G$.  Right-angled Coxeter groups are canonical examples of groups with nice geometric structures.  For example, the right-angled Coxeter group on a path of length at least 3 can be realized as a Fuchsian group which acts geometrically on strict subset of $\mathbb{H}^2$. Hence any $\CAT(0)$ boundary of such a right-angled Coxeter group is either 2 points or a Cantor set.

A technique for finding index two subgroups of a right-angled Coxeter group that features heavily in the current paper is doubling. Let $v\in S$ and define $D_v\Gamma$  to be the graph obtained from $\Gamma$ by gluing two copies of $\Gamma$ along the star of $v$ then deleting $v$ and its open star in the new graph. We call $D_v\G$  {\it the double of $\G$ over~$v$}. If a vertex $s$ of $\Gamma$ is not in the link $\Lk(v)$ we use $s'$ to denote its double in $D_v\Gamma$. The following lemma is a folk result analogous to Example 1.4 of Bestvina--Kleiner--Sageev~\cite{BKS08} (which is in the setting of right-angled Artin groups). For the sake of brevity we will not include a proof here.

\begin{lemma}[Doubling Lemma]\label{lem:doubling}
Assume $\Gamma$ is a finite simplicial graph, let $v$ be a vertex of $\Gamma$, and set $\Delta=D_v\Gamma$. Then $W_{\Delta}$ is an index two subgroup of $W_{\Gamma}$.
\end{lemma}

Associated to $\G$ is proper piecewise Euclidean $\CAT(0)$ complex $\Sigma_\G$ called the {\it Davis--Moussong} complex on which $\G$ acts geometrically. The space $\Sigma_\G$ is canonically constructed based solely on the combinatorial data of $\G$ (See~\cite{DavisBook} for the details of this construction.)

%%%%%%%%%%%%%%%%%%%%%%%%%%%%

%%%%%%%%%%%%%%%%%%%%%%%%%%%%%%%%%%%%%%%%

\subsection{Graph terminology}

Given a graph $\G$, an edge subdivision operation consists of adding a valence two vertex in the interior of an edge of $\G$.   A \emph{$\G$ subdivision} is a graph obtained from $\G$ by a (possibly trivial) sequence of edge subdivision operations.   A subgraph $\L$ of $\G$ is said to be 
\emph{induced} if every pair of vertices of $\L$ which are adjacent in $\G$ are also adjacent in 
$\L$ (i.e.~if $u$ and $w$ are vertices of $\L$, and $\G$ contains the edge $[u, v]$, then $\L$ does too). An essential vertex of $\G$ is any vertex of valence at least 3.  Vertices of valence two are called \emph{non-essential}.  A \emph{branch} of $\G$ is 
 an embedded path between essential vertices of the graph.  It contains its endpoints, but does not contain any other essential vertices.  The branch between a pair of essential vertices $x$ and $y$ will be denoted by $[x, y]$.  We will sometimes also use this interval notation for paths which are not necessarily branches, when there is no ambiguity.  
 
 A \emph{cycle} in $\G$ is an embedded loop.  We will denote cycles either by the essential vertices or by the paths that they visit.  For example, 
 if a cycle  passes through the essential vertices $v_1, \dots v_n$ in order, (so that $[v_i, v_{i+1}]$ is a branch for each $i$ (mod $n$),  then we will denote it by $(v_1, \dots v_n)$.  (We only use this notation for graphs in which every pair of essential vertices is connected by at most one branch.) 
On the other hand, if a cycle passes through embedded paths 
$\alpha_1, \dots, \alpha_n$ in $\G$ such that for each 
$i$ mod $n$, the terminal vertex of $\alpha_i$ is equal to the initial point of $\alpha_{i+1}$, then we denote it by $(\alpha_1, \dots, \alpha_n)$.  (Here we are not assuming that the $\alpha_i$ are branches.)

A graph is \emph{inseparable} if it is connected, has no separating complete subgraph, no cut pair, and no separating complete subgraph suspension. Obviously, a triangle-free graph is inseparable if and only if it is connected, has no separating vertex, no separating edge, no cut pair, and no separating vertex suspension.  This is equivalent to  the corresponding right-angled Coxeter group not splitting over a finite group or a virtually cyclic group \cite{P5, Sta68}.

%%%%%%%%%%%%%%%%%%%%%%%%%%%%%%%%%
\section{Graph theoretic results} \label{sec:graph} 
%%%%%%%%%%%%%%%%%%%%%%%%%%%%%%%%%

In this section we prove some graph theoretic results which are used in the next section to prove Theorem~\ref{thm:main1}. 

Let $\L\subset \G$ be a $\K$ (resp.~$K_5$) subdivision.  We say that a  vertex of $\L$ is 
 \emph{$\L$-essential} if it has valence bigger than 2 in $\L$, and \emph{$\L$-non-essential} if its valence in $\L$ is 2.  Note that given a vertex of $\L$, its valence in $\G$ could be higher than its valence in $\L$ (so in particular, a $\L$-non-essential vertex could be an essential vertex of $\G$).  
If $\L\subset \G$ is a $\K$ subdivision, a
\emph{vertex partition} for $\L$ is a partition of the $\L$-essential vertices into two sets of three vertices, such that every vertex in the first set is connected to every vertex of the second set by a branch of $\L$. 
By a \emph{shortest} graph with a given property, we will mean a graph having the fewest edges with that property. 

Kuratowski's Theorem says that a graph is planar if and only if it contains either a $K_5$ subdivision or a $\K$ subdivision. We begin with the following lemma, which will enable us to ignore the $K_5$ case when we are trying to establish the non-planarity of boundaries of right-angled Coxeter groups defined 
by non-planar graphs.

\begin{lemma}\label{lem:K5} 
Let $\G$ be a triangle-free graph which contains a $K_5$ subdivision $\Lambda$.  Then 
either $\G$ or the double of $\G$ over some vertex contains a $\K$ subdivision. 
\end{lemma}

\begin{proof}
Choose a shortest $K_5$ subdivision $\L$ in $\G$. 
Let $a, b, c, d, e$ be the $\L$-essential vertices.  
Since $\L$ is shortest, given any $\L$-essential vertex, say $a$, there cannot be a $\G$-edge between $a$ and some a vertex $x$ on a branch of $\L$ incident to $a$, unless $x$ is adjacent to $a$ in $\L$. 

Next suppose there is a $\G$-edge between some $\L$-essential vertex and some $\L$-non-essential vertex on a branch of $\L$ disjoint from it. (See Figure~\ref{fig:K5K33}.)  For definiteness, say there is an edge from $a$ to the $\L$-non-essential vertex  
$f$ in the interior of $[d,e]$.  Then there is a $\K$ subdivision with vertex partition $\{a, e, d\}$ and $\{b, c, f\}$, as shown in Figure~\ref{fig:K5K33}.  This completes the proof of the lemma in this case. 
\begin{figure}[h!]
\begin{center}
\begin{overpic}[scale=0.7]{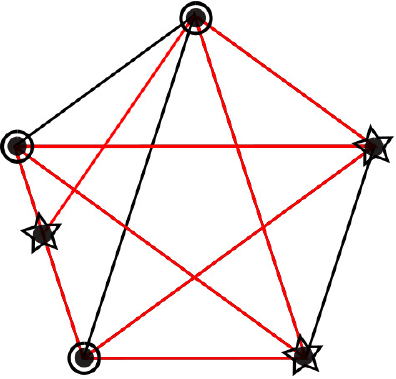}
\put(45, 98){\small $a$}
\put(102, 55){\small $b$}
\put(84, 0){\small $c$}
\put(7, 0){\small $d$}
\put(-8, 55){\small $e$}
\put(-3, 32){\small $f$}
\end{overpic}
\end{center}
\caption {}
\label{fig:K5K33}
\end{figure}

From the previous two paragraphs, we may assume for the remainder of the proof that if a $\L$-essential vertex $a$  is $\G$-adjacent to a vertex $x$ of $\L$, then $a$ is adjacent to $x$ in $\L$.

We claim that by relabeling the $\L$-essential vertices if necessary, we may assume  that $a$ is not adjacent to vertices $b$ and $c$. i.e.~the branches $[a, b]$ and $[a, c]$ of $\L$ are subdivided.  To see this, note that if there exist two $\L$-essential vertices not adjacent to $a$, then we can simply relabel these as $b$ and $c$.  Otherwise, $a$ is adjacent to at least three $\L$-essential vertices, say $b, c, d$.  Now since $\G$ is triangle-free, each of $[b, c]$, $[c, d]$ and $[b, d]$ is subdivided, and we can rename $b$ to $a$, and $d$ to $b$.  This proves the claim.

We will now produce a 
$\K$ subdivision in $D_a\G$, assuming that $b$ and $c$ are not adjacent to $a$.   
By our assumption in the third paragraph, the link of $a$ intersected with $\L$ consists of exactly four vertices, one on each branch incident to $a$.  The vertices $d$ and $e$ could be among these.  
Let $\bar \G$ and $\bar \G'$ both denote $\G$ minus the open star of $a$, and recall that $D_a\G$ is obtained by identifying $\bar \G$ and $\bar\G'$ along the link of $a$. For each vertex $x$ of $\G$, let $x$ and $x'$ be the corresponding vertices in $\bar \G$ and  $\bar \G'$ respectively. 
Since $b$ and $c$ are not adjacent to $a$ in $\G$, we have that $b' \neq b$ and $c' \neq c$ in $D_a\G$.   However, we could have $d=d'$ or $e=e'$.

\begin{figure}[h!]
\begin{center}
\begin{overpic}[scale=0.7 ]{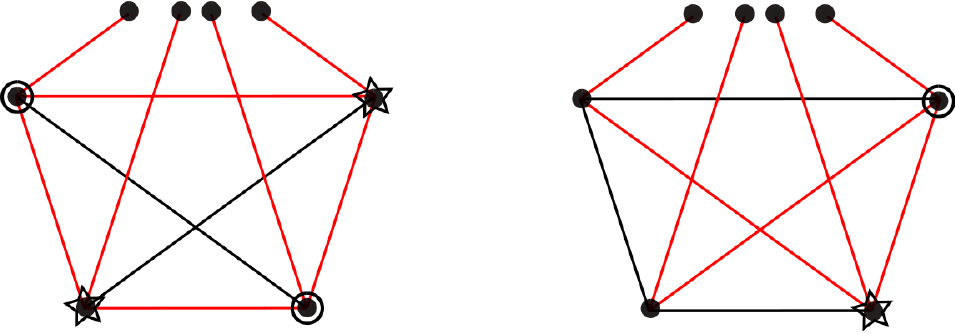}
\put(-12, 12){\small $\bar \G$}
\put(107, 12){\small $\bar \G'$}
\put(42, 22){\small $ b$}
\put(35, 0){\small $ c$}
\put(3, 0){\small $ d$}
\put(-4, 22){\small $ e$}
\put(12, 36){\small $ e_1$}
\put(17, 36){\small $ d_1$}
\put(21.5, 36){\small $c_1$}
\put(26, 36){\small $ b_1$}
\put(71, 36){\small $ e_1$}
\put(76, 36){\small $ d_1$}
\put(80.5, 36){\small $c_1$}
\put(86, 36){\small $ b_1$}
%\put(45, 98){\small $ a'$}
\put(101, 22){\small $ b'$}
\put(94, 0){\small $ c'$}
\put(62, 0){\small $ d'$}
\put(56, 22){\small $ e'$}
\end{overpic}
\end{center}
\caption {The graph $D_a\G$ is obtained by identifying $\bar\G$ and $\bar \G'$ along 
$b_1, c_1, d_1$, and $e_1$.  Thus for example, the path from $d$ to $b'$ in $D_a\G$ consists of the branch $[d, d_1]$ in $\bar\G$ followed by the branches $[d_1, d']$ and $[d', b']$. 
When $a$ is adjacent to $d$ (respectively $e$), then $d=d_1=d'$ 
(respectively $e=e_1=e'$).  }
\label{fig:K5double}
\end{figure}
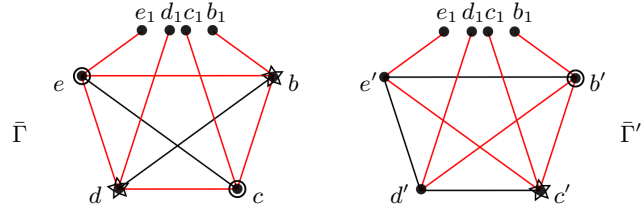

We claim that 
there is a $\K$ subdivision in $D_a\G$ with vertex partition $\{b, d, c'\}$ and $\{c, e, b' \}$.  This is shown in Figure~\ref{fig:K5double} in the case when $d \neq d'$ and $e \neq e'$.  When $d=d'$, the path from $b'$ to $d$ in Figure~\ref{fig:K5double} is replaced by the branch in $\bar \G'$ from $b'$ to $d'=d$.  Similarly, when $e=e'$, the path from $c'$ to $e$ shown in Figure~\ref{fig:K5double} is replaced by the branch in $\bar \G'$ from $c'$ to $e'=e$.  
Since the link of $a$ in $\L$ consists of exactly four vertices, the subdivision constructed above is embedded in $D_a\G$.  
\end{proof}

The $\K$ subdivision present in a non-planar graph $\G$ may not be an induced subgraph.  
In Proposition~\ref{prop:lambdas} we show that by successively doubling $\G$ along vertices finitely many times, one can find a subgraph that is either an  induced $\K$ subdivision, or one of two specific graphs.

\begin{proposition}\label{prop:lambdas}
Let $\G$ be a triangle-free graph which contains 
a $\K$
 subdivision. 
 Then there exists a graph $\G'$ 
obtained from $\G$ by  a sequence of doubling moves, and a $\K$ subdivision
$\Lambda' \subseteq \G'$ such that either $\L'$ is induced or  the subgraph of $\G'$ induced by $\L'$ is 
one of the graphs in Figure~\ref{fig:lambdas}.
\end{proposition}

\begin{figure}[h!]
\begin{center}
\begin{overpic}[scale=0.7]{fig-lambdas1.pdf}
\put(75, 75){\small $ a$}
\put(79, -4){\small $ b$}
\put(-6, 35){\small $ c$}
\put(22, 75){\small $ x$}
\put(102, 35){\small $ y$}
\put(18, -4){\small $ z$}
\end{overpic}
\hspace{1.5cm}
\begin{overpic}[scale=0.8]{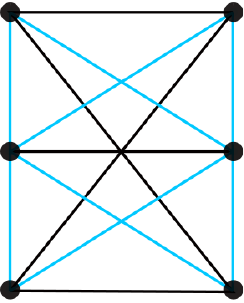}
\put(-10, 94 ){\small $ a$}
\put(84, 94){\small $ x$}
\put(-10, 50){\small $ b$}
\put(84, 50){\small $ y$}
\put(-10, 0){\small $ c$}
\put(84, 0){\small $ z$}
\end{overpic}
\end{center}
\caption {The figure shows the two possible graphs induced by $\L'$ in $\G'$ in Proposition~\ref{prop:lambdas}.  In both pictures, $\L'$ is the $\K$ subdivision with vertex partition 
$\{a, b, c\}$ and $\{x, y, z\}$.  The  black paths may  be subdivided, while the blue ones are edges.  There are no edges of
$\G' \setminus \Lambda'$ connecting any pair of vertices in the graphs shown.}
\label{fig:lambdas}
\end{figure}
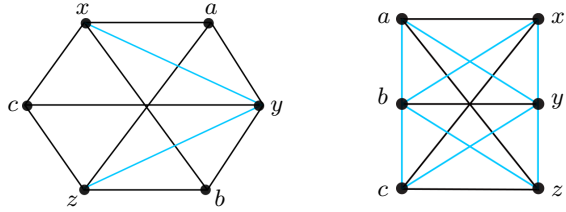

We begin by introducing some terminology and proving several lemmas to be used in the proof.  
In all of the proofs below, we assume that $\L$ has $\L$-essential vertex sets $\{a, b, c\}$ and $\{x, y, z\}$. 
A \emph{branch} of $\L$ is the unique path between a pair of $\L$-essential vertices which does not pass through any additional $\L$-essential vertices.  We assume that there are branches of $\L$ connecting each of $a, b,$ and $c$ to each of $x, y$, and $z$, and 
we will denote the branch between $\L$-essential vertices, say $a$ and $x$, by $[a, x]$.  We will sometimes also use this interval notation to denote sub-paths of branches or edges.  A branch or path $[s,t]$ will always include its endpoints $s$ and $t$.  Two branches will be called \emph{adjacent} if they share an $\L$-essential vertex endpoint.

If $\L \subset \G$ is not induced, we define a \emph{bad edge} of $\L$ to be an edge in 
$E(\G) \setminus E(\L)$ whose endpoints are both vertices of $\L$, and we define $B(\L)$ to be the number of bad edges of $\L$.  Observe that $\L$ is induced if and only if $B(\L)=0$.

\subsubsection*{Idea of the proof.} 
If the $\K$ subdivision $\L$ is not induced, it has a non-trivial set of bad edges.  If one doubles $\G$ over the endpoint $v$ of a bad edge of $\L$, then  that edge disappears in $D_v\G$.  So the general strategy is to double over endpoints of bad edges of $\L$ and to find a new $\K$ subdivision in the double which has fewer bad edges than $\L$.  Then after finitely many steps we end up with either an induced $\K$ subdivision or one of the graphs in Figure~\ref{fig:lambdas}.

The double of $\L$ in $D_v\G$ typically has many more bad edges than $\L$ itself, and as a result, finding a $\K$ subdivision in $D_v\G$ with fewer bad edges than $\L$  can be a nontrivial feat.  To aid this process, we do two things.  
Firstly, we start with a shortest subdivision $\L$, and we show in Lemma~\ref{lem:adjacent} that this restricts the types of bad edges that may occur in $\L$.  (Types of bad edges may be differentiated based on whether they connect disjoint or adjacent branches, and whether their endpoints are $\L$-essential or not.) 
Secondly, we choose the order of vertices to double over carefully.  

Initially, we focus on reducing the number of bad edges which have at least one endpoint a $\L$-non-essential vertex.  In particular, we
show in Lemmas~\ref{lem:double} and~\ref{lem:non-ess}, that if $v$ is a $\L$-non-essential vertex 
which is the endpoint of a bad edge of $\L$, then unless the edges incident to $v$ have a specific configuration (shown in Figure~\ref{fig:vbwu}), the double $D_v\G$ does contain a $\K$ subdivision with fewer bad edges than $\L$.  Then, in the proof of Proposition~\ref{prop:lambdas}, we show how to deal with the problematic configuration given in the statement of Lemma~\ref{lem:non-ess}, Figure~\ref{fig:vbwu}. This eliminates all bad edges which have at least one $\L$-non-essential vertex as an endpoint.  Also in  the proof of Proposition~\ref{prop:lambdas},
we resolve the case in which all bad edges have 
$\L$-essential vertices as both their endpoints.  In is in this case that the graphs in Figure~\ref{fig:lambdas} arise.

\bigskip
We begin with Lemma~\ref{lem:adjacent}, which puts restrictions on the type of bad edges one could see in a shortest $\K$ subdivision with a given number of bad edges.

\begin{lemma}\label{lem:adjacent}
Let $\G$ be a triangle-free graph containing a $\K$
 subdivision $\L$.  
If $\L$ is shortest among all $\K$
 subdivisions with at most $B(\L)$ bad edges, then $\L$ has no bad edges of the following types:
\begin{enumerate}
\item Bad edges with both endpoints on a single branch of $\L$.
\item  
Bad edges  connecting non-$\L$-essential vertices on adjacent branches.  
 \end{enumerate}
 \end{lemma}
\begin{proof}
Item (1) is obvious. For (2),   
suppose $\L$ has an edge connecting 
non-$\L$-essential vertices $v$ and $w$ on two adjacent branches, say 
$[a, x]$ and $[a, y]$ (see Figure~\ref{fig:adjacent}).  Since $\G$ is triangle-free, the path from $a$ to one of these vertices, say $w$, must be subdivided.  
Then there is a shorter $\K$ subdivision with at most $B(\L)$ bad edges
(having vertex partition $\{x, y, z\}$
and $\{v, b, c\}$)
as shown in Figure~\ref{fig:adjacent}.  This is a contradiction. 
\end{proof}
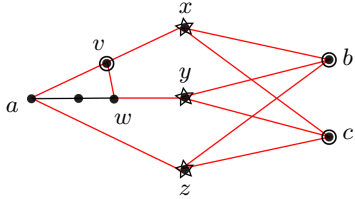
\begin{figure}[h!]
\begin{center}
\begin{overpic}[scale=0.6]{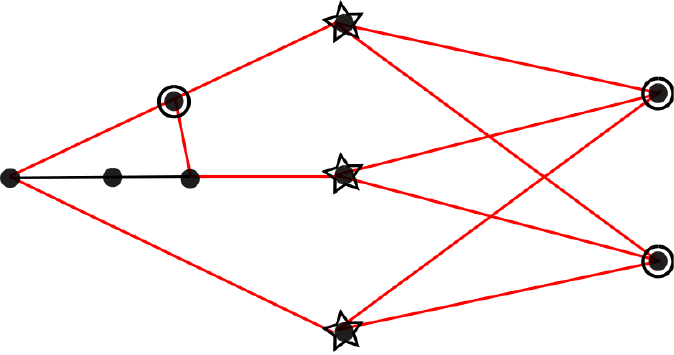}
\put(-7, 22){\small $ a$}
\put(102, 37){\small $ b$}
\put(102, 13){\small $ c$}
\put(49, 54){\small $ x$}
\put(49, 33){\small $ y$}
\put(49, -5){\small $ z$}
\put(21, 42){\small $ v$}
\put(28, 18){\small $ w$}
\end{overpic}
\end{center}
\caption {Consider the $\K$ subdivision shown, with $\L$-essential vertex sets $\{x, y, z\}$
and $\{v, b, c\}$.  It is shorter than $\L$ and each of its bad edges is already a bad edge of $\L$.} 
\label{fig:adjacent}
\end{figure}

The next step  is to begin doubling over 
$\L$-non-essential vertices, and to find $\K$ subdivisions with fewer bad edges in the double.  
We now know that bad edges between two $\L$-non-essential vertices must go between two disjoint branches.  
  Lemma~\ref{lem:edge} gives a useful consequence of the existence of such a bad edge.  This will be used in Lemma~\ref{lem:double}, where we find a $\K$ subdivision in the double over an endpoint of such an edge.

\begin{lemma}\label{lem:edge}
Let $\G$ be a triangle-free graph, 
 and let $\L$ be shortest among all $\K$
 subdivisions with at most $B(\L)$ bad edges.
 Suppose there is a bad edge connecting 
 non-$\L$-essential vertices on disjoint branches $\alpha$ and $\beta$ of $\L$.  Then the unique branch of $\L$ disjoint from $\alpha$ and $\beta$ is an edge. 
\end{lemma}
\begin{proof}
Assume without loss of generality that there is an edge connecting non-$\L$-essential vertices $v$ and $w$ on branches $[a, x]$ and $[c, z]$ respectively.  
If $[b,y]$, the unique branch of $\L$ disjoint from $[a, x]$ and $[c, z]$,  
is not an edge, then one obtains a shorter $\K$
 subdivision with at most $B(\L)$ bad edges,
(having vertex partition 
 $\{v, c, z\}$
and $\{w, a, x\}$) 
as shown in Figure~\ref{fig:edge}. 
\end{proof}

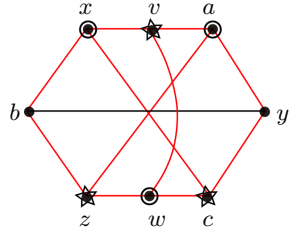
\begin{figure}[h!]
\begin{center}
\begin{overpic}[scale=0.7]{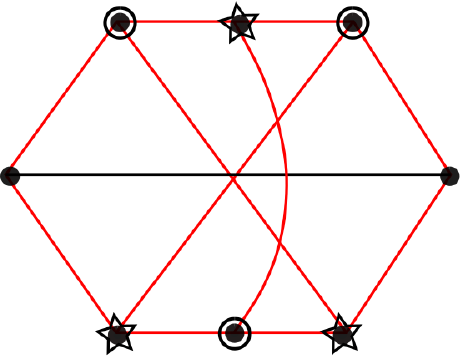}
\put(22, 79){\small $ x$}
\put(72,79){\small $ a$}
\put(102, 36){\small $ y$}
\put(-6, 36){\small $ b$}
\put(22, -7){\small $ z$}
\put(72, -7){\small $ c$}
\put(50, 79){\small $ v$}
\put(50, -7){\small $ w$}
\end{overpic}
\end{center}
\caption{The $\K$ subdivision shown, with $\L$-essential vertex sets 
 $\{v, c, z\}$
and $\{w, a, x\}$, has at most $B(\L)$ bad edges.  It
is shorter than $\L$ unless $[b,y]$ is an edge. 
}
\label{fig:edge}
\end{figure}

The next two lemmas deal with finding $\K$ subdivisions in doubles $D_v\G$, where $v$ is a $\L$-non-essential vertex which is the endpoint of a bad edge. 
Lemma~\ref{lem:double} gives a criterion on $v$ which guarantees that $D_v\G$ has a 
$\K$ subdivision with fewer bad edges.  This will be used in Lemma~\ref{lem:non-ess}.

\begin{lemma}\label{lem:double}
Let $\G$ be a triangle-free graph,
 and let $\L$ be shortest among all $\K$
 subdivisions with at most $B(\L)$ bad edges.
Let $v$ be a non-$\L$-essential vertex on a branch $\alpha$ of $\L$ such that $v$ is an endpoint of a bad edge.  If there is a branch $\beta$ which is disjoint from $\alpha$ such that $\beta$ does not intersect the link of $v$, then there exists a $\K$ subdivision $\L' \subseteq D_v(\L)$ with $B(\L') < B(\L)$. 
\end{lemma}

\begin{proof}
Assume without loss of generality that 
$\alpha=[a, x]$ and that $\beta=[c,z]$ is disjoint from the link of $v$. 
See Figure~\ref{fig:v-alpha}.
We
know by Lemma~\ref{lem:adjacent}(2) that the interiors of the branches $[a, z]$ and $[x, c]$ are disjoint from the link of $v$, and by Lemma~\ref{lem:adjacent}(1) that the link of $v$ intersects $[a,x]$ in exactly two vertices, $u$ and  
$w$, with $u$ possibly equal to $a$ and $w$ possibly equal to $x$.  
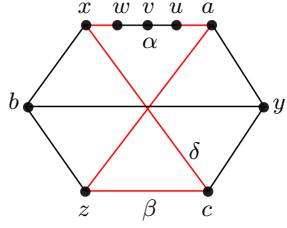
\begin{figure}[h!]
\begin{center}
\begin{overpic}[scale=0.7]
{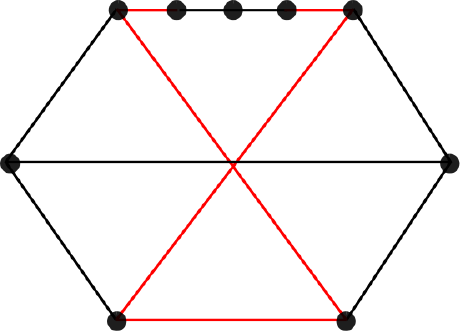}
\put(22, 75){\small $ x$}
\put(72,75){\small $ a$}
\put(101, 36){\small $ y$}
\put(-6, 36){\small $ b$}
\put(22, -7){\small $ z$}
\put(72, -7){\small $ c$}
\put(36, 75){\small $ w$}
\put(48, 75){\small $ v$}
\put(59, 75){\small $ u$}
\put(48, 61){\small $\alpha$}
\put(48, -7){\small $\beta$}
\put(67, 16){\small $\delta$}
\end{overpic}
\end{center}
\caption{The link of $v$ consists of $u$, $w$ and possibly some vertices in $[z, b] \cup [b, y]\cup [y, c]$.
The path $\delta$ is shown in red.}
\label{fig:v-alpha}
\end{figure}
It follows that the path $\delta$ from $u$ to $w$ which consists of the concatenation of $[u, a], [a, z], [z, c], [c, x], [x, w] $ intersects the link of $v$ only in $u$ and $w$.

Now let $\bar \G$ and $\bar \G'$ be two copies of $\G$ with the open star of $v$ removed.  The double $D_v\G$ is formed by identifying $\bar \G$ and $\bar \G'$ along the copy of the link of $v$ in each.  
Define $\bar\L \subseteq \bar\G$ to be the copy of $\L$ minus the open star of $v$ in $\bar \G$, and note that $\bar\L$
would be a $\K$ subdivision if we added a path between $u$ and $w$ which is disjoint from $\bar\L$.
There is a copy of the path $\delta$ constructed above in $\G'$, which intersects $\bar\G$ (and 
hence $\bar \L$) only at $u$ and $w$. Let $\gamma$ be the shortest path in $\bar\G'$ between $u$ and $w$ which which intersects $\bar\G$ only at $u$ and $w$.
Form $\L'$ by identifying $\bar\L$ and $\gamma$ along $\{u, w\}$.  (This is shown in Figure~\ref{fig-double} in the case $\gamma = \delta$.)
Then $\L'$ is a $\K$ subdivision in $D_v\G$.

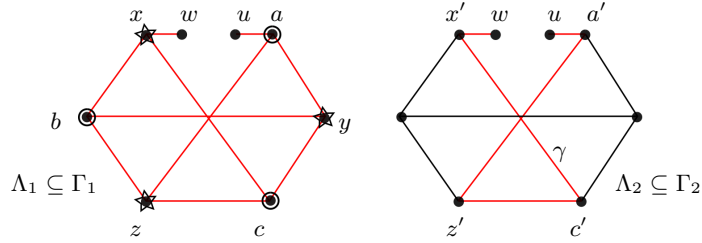
\begin{figure}[h!]
\begin{center}
\begin{overpic}[scale=0.7]{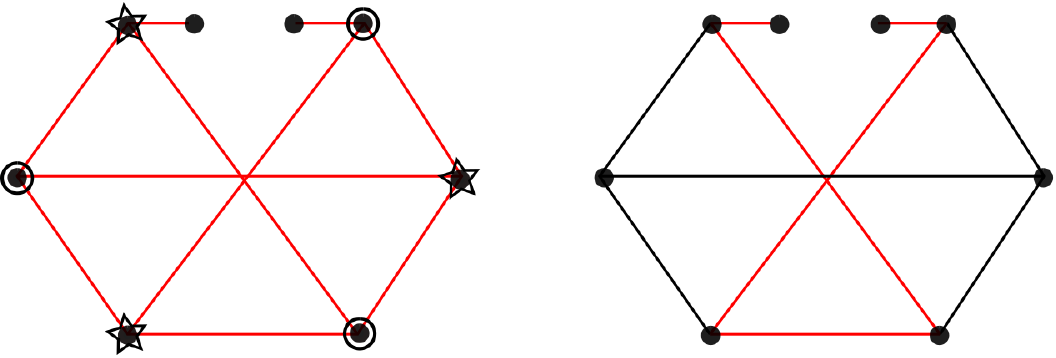}
\put(-12, 4){\small $ \L_1\subseteq \G_1$}
\put(95, 4){\small $ \L_2 \subseteq \G_2$}
\put(9, 34){\small $ x$}
\put(18, 34){\small $ w$}
\put(28, 34){\small $ u$}
\put(34, 34){\small $ a$}
\put(65, 34){\small $ x'$}
\put(73, 34){\small $ w$}
\put(83, 34){\small $ u$}
\put(90, 34){\small $ a'$}
\put(46, 15){\small $ y$}
\put(-5, 15){\small $ b$}
\put(9, -4){\small $ z$}
\put(31, -4){\small $ c$}
\put(65, -4){\small $ z'$}
\put(87, -4){\small $ c'$}
\put(84, 10){\small $ \gamma$}

\end{overpic}
\end{center}
\caption{The two graphs shown are identified along $\mathrm{Lk}(v)$ in $D_v\G$.  In particular, they are identified along $u$ and $w$, and possibly some additional vertices on 
$[z, b] \cup [b, y]\cup [y, c]$.
The red graph on the left is $\bar\L$.  The red path shown on the right is $\gamma$, in the case that $\gamma= \delta$.   The choice of $\gamma$ ensures that no vertex on it is identified with a vertex in $\bar \G$.    
}
\label{fig-double}
\end{figure}

If $e$ is a bad edge of $\L'$, our choice of $\gamma$ implies that the endpoints of $e$ are in $
\bar\G$.  Thus there is a bad edge of $\L$ in $\G$ that corresponds to $e$.  On the other hand, there is at least one bad edge of $\L$ incident to $v$, for which there is no corresponding bad edge of $\L'$.  Thus $B(\L') < B(\L)$. 
\end{proof}

The following lemma shows that if $v$ is $\L$-non-essential and is the endpoint of a bad edge of $\L$, then $D_v\G$ does contain a $\K$ subdivision with fewer bad edges than $\L$, except possibly in one particular situation.  
(This situation is specified in conditions (1) and (2) of the lemma).

\begin{lemma}\label{lem:non-ess}
Let $\G$ be a triangle-free graph,
 and let $\L$ be shortest among all $\K$
 subdivisions with at most $B(\L)$ bad edges.
Let $v$ be a non-$\L$-essential vertex on a branch $\alpha$ of $\L$ such that $v$ is an endpoint of a bad edge.
 Then either there exists a $\K$ subdivision $\L' \subseteq D_v(\L)$ with $B(\L') < B(\L)$ or 
both of the following statements hold.  (See Figure~\ref{fig:vbwu}.) 
 \begin{enumerate}
 \item The vertex $v$ is adjacent to exactly one 
 $\L$-essential vertex $s$ of $\L$ which is not an endpoint of $\alpha$, and to at least one non-$\L$-essential vertex on each of the two branches that are disjoint from both $\alpha$ and $s$.  
 
 \item The branch $\alpha$ consists of exactly the two edges incident to $v$.

\medskip

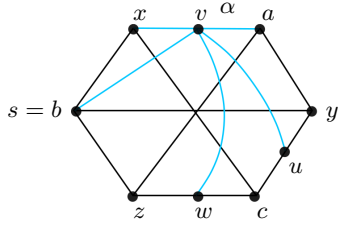
\begin{figure}[h!]
\begin{center}
\begin{overpic}[scale=0.7]{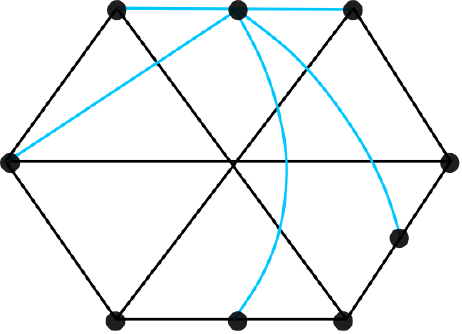}
\put(25, 74){\small $ x$}
\put(50, 74){\small $ v$}
\put(77, 74){\small $ a$}
\put(-26, 35){\small $ s= b$}
\put(103, 35){\small $ y$}
\put(88, 12){\small $ u$}
\put(75, -6){\small $ c$}
\put(50, -6){\small $ w$}
\put(25, -6){\small $ z$}
\put(60, 77){\small $ \alpha$}
\end{overpic}
\end{center}
\caption{This illustrates the configuration from Lemma~\ref{lem:non-ess} in the case that $v$ lies on $[a,x]$, and $v$ is adjacent to $b$.  Then the lemma says that  $v$ must be adjacent to at least one vertex on each of $[c, y]$ and $[c,z]$, and that $[v, x]$ and $[v,a]$ are edges. 
}
\label{fig:vbwu}
\end{figure}

 \end{enumerate}
\end{lemma}
\begin{proof}
Without loss of generality, assume $v$ lies on $\alpha=[a, x]$.  
First we consider the case that $v$ is not adjacent to any of the $\L$ essential vertices $b, c, y,$ or $z$.   In this case we show that $D_v\G$ contains a $\K$ subdivision $\L' \subseteq D_v(\L)$ with $B(\L') < B(\L)$. 
By hypothesis, there is a bad edge incident to $v$, and by Lemma~\ref{lem:adjacent}, its other endpoint has to be on a branch disjoint from $\alpha$, say (without loss of generality)  
$[c,z]$.  Then by Lemma~\ref{lem:edge}, the branch $[b,y]$ is an edge.  
Since by assumption $v$ is not adjacent to $b$ or $y$, we may apply 
Lemma~\ref{lem:double} with $\beta = [b,y]$, to conclude that 
 there exists a $\K$ subdivision $\L' \subseteq D_v(\L)$ with $B(\L') < B(\L)$. 

Thus we may assume that $v$ is adjacent to at least one of $b, c, y,$ and $z$.  Now we analyze a few cases. 
If $v$ is adjacent to both $b$ and $c$, then one obtains a shorter $\K$ subdivision with at most $B(\L)$ bad edges, as shown on the left in 
Figure~\ref{fig:vbc-vbwu}, which is a contradiction, since $\Lambda$ was chosen to be shortest.  Thus $v$ is adjacent to at most one of $b$ and $c$, and similarly, $v$ is adjacent to at most one of $y$ and $z$. 

Suppose $v$ is adjacent to exactly one from each pair, say $b$ and $y$.  Then since $\G$ is triangle-free, 
$[b,y]$ is not an edge.  Applying Lemma~\ref{lem:edge}, we conclude that $v$
 is not adjacent to any vertex in $[c, z]$.  Then by Lemma~\ref{lem:double}, with $\beta = [c, z]$, 
 there exists a $\K$ subdivision $\L' \subseteq D_v(\L)$ with $B(\L') < B(\L)$.

\begin{figure}[h!]
\begin{center}
\begin{overpic}[scale=0.7]{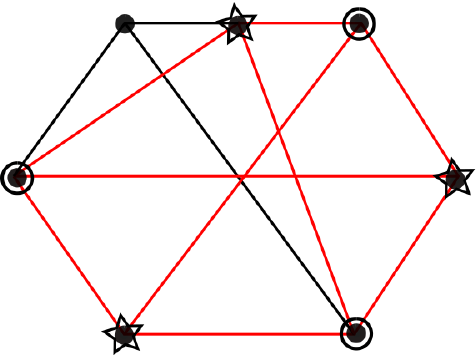}
\put(25, 76){\small $ x$}
\put(50, 76){\small $ v$}
\put(77, 76){\small $ a$}
\put(-6, 35){\small $ b$}
\put(103, 35){\small $ y$}
\put(75, -6){\small $ c$}
\put(25, -6){\small $ z$}

\end{overpic}
\hspace{1.5cm}
\begin{overpic}[scale=0.7]{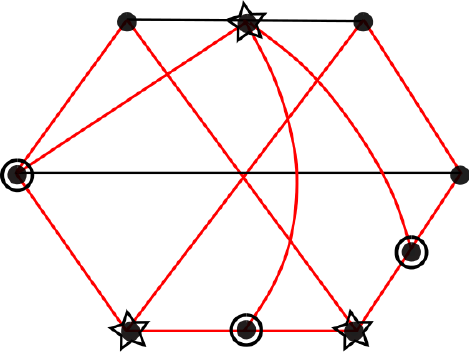}
\put(25, 76){\small $ x$}
\put(50, 76){\small $ v$}
\put(77, 76){\small $ a$}
\put(-6, 35){\small $ b$}
\put(103, 35){\small $ y$}
\put(88, 12){\small $ u$}
\put(75, -6){\small $ c$}
\put(50, -6){\small $ w$}
\put(25, -6){\small $ z$}
\end{overpic}
\end{center}
\caption{ If $v$ is connected to both $b$ and $c$, then one obtains the red graph on the left.  It is shorter than $\L$ and has at most $B(\L)$ bad edges.  If $v$ is adjacent to $b, w, $ and $u$ as shown, and if one of $[a, v]$ and $[v, x]$ is not an edge, then the red graph on the right is shorter than $\L$ and has at most $B(\L)$ bad edges.  
}
\label{fig:vbc-vbwu}
\end{figure}
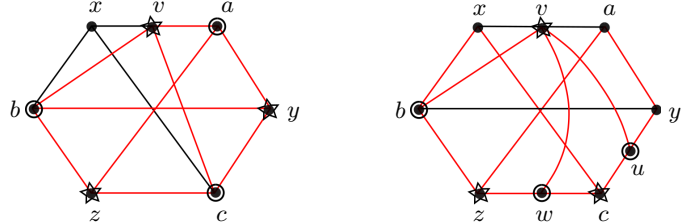

We are left with the case that $v$ is adjacent (via a bad edge)  to exactly one $\L$-essential vertex, say $b$. 
As before, if $\mathrm{Lk}(v)$ fails to intersect one of the branches $[c, y]$ and $[c, z]$, then applying  
Lemma~\ref{lem:double}, we would find the desired $\L'$ in $D_v\G$.  
If not, then $v$ is adjacent to 
  non-$\L$-essential vertices $u$ and $w$ on   $[y,c]$ and $[z,c]$ respectively, i.e.~condition (1) in the statement of the lemma holds.  
Finally, if we have the configuration in (1), but one of $[x, v]$ and $[a, v]$ is not an edge, we get a shorter $\K$ subdivision as shown on the right in Figure~\ref{fig:vbc-vbwu}.
The above reasoning shows that either $D_v\G$ contains a $\K$ subdivision $\L' \subseteq D_v(\L)$ with $B(\L') < B(\L)$ or (1) and (2) both hold simultaneously. 
\end{proof}

Using the Lemma~\ref{lem:non-ess}, we can keep doubling over $\L$-non-essential vertices and finding 
$\K$ subdivisions with fewer bad edges until we either run out of $\L$-non-essential vertices which are endpoints of bad edges, or get to the point where {\it every} such $\L$-non-essential vertex has the configuration in Figure~\ref{fig:vbwu}.  In particular, by applying Lemma~\ref{lem:non-ess} to $w$ and $u$ (from Figure~\ref{fig:vbwu}) we obtain strong restrictions on the graph, which enable us to find an induced $\K$ subdivision.  Finally, we show that when all of the bad edges have $\L$-essential vertex endpoints, then either the graph is one of the graphs in Figure~\ref{fig:lambdas} or some double contains a $\K$ subdivision with fewer bad edges.  
This is all put together in the following proof.

\begin{proof}[Proof of Proposition~\ref{prop:lambdas}]
Choose a $\K$ subdivision $\Lambda \subseteq \G$ such that $\Lambda$ is shortest among all 
$\K$ subdivisions with at most $B(\L)$ bad edges.

\medskip
\noindent
\emph{Claim: }  If $\L$ is not induced, one of the following holds:
\begin{enumerate}
\item[(i)]  
$\L$ induces one of the graphs in Figure~\ref{fig:lambdas}. 
\item[(ii)] 
For some vertex $v$, the double $D_v\G$ 
contains a $\K$ subdivision $\L'$ such that $B(\L') < B(\L)$.  
\end{enumerate}

 Before proving the claim, we explain why it is sufficient to complete the proof.  
Given a $\K$ subdivision $\L$, if it is induced or if (i) holds, i.e.~if it induces one of the graphs in 
Figure~\ref{fig:lambdas}, then we are done.  Otherwise, (ii) holds.  We take $\L_2$ to be the shortest $\K$ subdivision in $\G_2=D_v\G$ with at most $B(\L')$ bad edges, where $D_v\G$ and $\L'$ are 
provided by (ii), and 
we repeat the argument with $\L_2$  and $\G_2$ instead of $\L$ and $\G$.  After finitely many steps we arrive at a pair $\L_n \subseteq \G_n$ which either satisfies (i) or such that  
 $B(\L_n)=0$, which means $\L_n$ is induced.  This proves the proposition. 
 
 \medskip
\emph{Proof of the claim:}  By Lemma~\ref{lem:adjacent}, 
$\L$ has no bad edges between any pair of vertices that lie on a single branch, or any pair of $\L$-non-essential vertices on adjacent branches.  

{\bf Case 1. Non-essential vertex on a bad edge.} Suppose there exists a $\L$-non-essential vertex $v$, say on the branch $[a, x]$,  which is the endpoint of a bad edge. Then by Lemma~\ref{lem:non-ess}, either (ii) in the claim above holds 
(in which case we are done)
or we may assume that $\L$ has the configuration specified by conditions (1) and (2) of Lemma~\ref{lem:non-ess}.  In the latter case, $[v, a]$ and $[v, x]$ are edges and we can assume, by re-labeling if necessary, that the edges incident to $v$ guaranteed by (1) are as shown 
in Figure~\ref{fig:vbwu}.  In particular, $v$ is adjacent to $b$, and to $\L$-non-essential vertices 
$w$ on $[c, z]$ 
and $u$ on $[c, y]$.  

Now apply Lemma~\ref{lem:non-ess} to $w$. 
If $D_w\G$ contains  a $\K$ subdivision $\L'$ such that $B(\L') < B(\L)$, we are done.  If not, 
we conclude (from Lemma~\ref{lem:non-ess}(1)) that $w$ is adjacent to exactly one of $a, b, x$ and $y$.   
(See the left side of Figure~\ref{fig:induced}.)
Since $\G$ is triangle-free, and $v$ is already adjacent to $a, x, b$ and $w$, we see that $w$ cannot be adjacent to any of $a, x, $ or $b$.  
Thus we conclude that $w$ is adjacent to $y$. By Lemma~\ref{lem:non-ess}(1), 
there is a $\L$-non-essential vertex $t$ on $[x, b]$
adjacent to $w$. Furthermore, $[w, c]$ and $[w, z]$ are edges by   Lemma~\ref{lem:non-ess}(2).
\begin{figure}[h!]
\begin{center}
\begin{overpic}[scale=0.7]{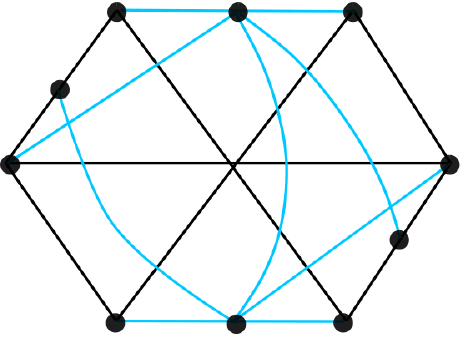}
\put(25, 74){\small $ x$}
\put(50, 74){\small $ v$}
\put(77, 74){\small $ a$}
\put(-8, 35){\small $ b$}
\put(103, 35){\small $ y$}
\put(88, 12){\small $ u$}
\put(75, -6){\small $ c$}
\put(50, -6){\small $ w$}
\put(25, -6){\small $ z$}
\put(4, 54){\small $t$}
\end{overpic}
\hspace{3cm}
\begin{overpic}[scale=0.7]{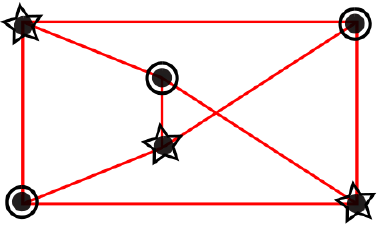}
\put(45, 45){\small $ x$}
\put(44, 9){\small $ c$}
\put(-3, 59){\small $ v$}
\put(100, 59){\small $ u$}
\put(-2, -8){\small $ w$}
\put(100, -8){\small $ t$}
\end{overpic}
\end{center}
\caption{The picture on the left shows the configuration obtained after Lemma~\ref{lem:non-ess} has been applied to $v$ and then to $w$.  Blue paths are edges. 
\\
After applying Lemma~\ref{lem:non-ess} to $u$, we have additional edges from $u$ to $z$ and $t$.  
Then the resultant graph contains  a $\K$ subdivision as shown in the picture on the right. 
}
\label{fig:induced}
\end{figure}

Applying similar reasoning to $u$, 
we conclude (from Lemma~\ref{lem:non-ess}(2)) that $[c, u]$, $[u, y]$ are edges and (from Lemma~\ref{lem:non-ess}(1)) that $u$ is adjacent to $z$ and to a $\L$-non-essential vertex $t'$ on the branch $[x, b]$. Finally, applying Lemma~\ref{lem:non-ess}(2) to the vertex $t$ from the previous paragraph, we conclude that $[t, x]$ and $[b, t]$ are edges.  It follows that $t=t'$, so that $u$ is adjacent to $t$.  

Then $\G$ contains 
a $\K$ subdivision as shown on the right in Figure~\ref{fig:induced}.  
All the branches of this $\K$ subdivision are edges, except possibly $[c, x]$.  
 If $e$ is a bad edge of this graph, then Lemma~\ref{lem:adjacent}(1) together with the triangle-free condition implies that $e$ must connect one of $t, u, v$ and $w$ to a $\L$-non-essential vertex on $[x,c]$.  However, since each of $t, u, v,$ and $w$ lies on a branch of $\L$ adjacent to $[c, x]$, Lemma~\ref{lem:adjacent}(2) implies that there are no bad edges of this kind.  Thus the $\K$ subdivision is induced. In particular, (ii) of the claim holds.  
This completes the proof of the claim in the case that there is at least one non-$\L$-essential vertex of $\L$ which is the endpoint of a bad edge.

{\bf Case 2.  Every bad edge has essential vertices}. It remains to consider the case that all endpoints of bad edges are $\L$-essential.  
Lemma~\ref{lem:adjacent} (1) implies that any bad edge has its both its endpoints in $\{x, y, z\}$ or both in $\{a, b, c\}$.  Moreover, since $\G$ is triangle-free, for each of these sets, there can be at most two edges connecting pairs of vertices in the set. 

First consider the case when one of the sides has exactly one edge (and the other side has zero, one or two edges).  Assume without loss of generality, that  there is a bad edge 
between $x$ and $y$, but none between either of these and $z$.  
Then $D_x\G$ 
contains a $\K$ subdivision $\L'$ such that $B(\L') < B(\L)$.  
If there are no bad edges with endpoints among $a', b', c'$, then the required $\K$ subdivision is 
as shown in Figure~\ref{fig:isolated}, and is induced.  If there are such bad edges, then in particular, these bad edges have endpoints between non-essential vertices of adjacent branches of the $\K$ subdivision in Figure~\ref{fig:isolated}, and they can be eliminated  using the procedure in the proof of Lemma~\ref{lem:adjacent} (2).  Then the resulting graph only has bad edges with endpoints 
among $a, b, c$, but we have eliminated the bad edge between $x$ and $y$.  So the resulting graph has fewer bad edges.  
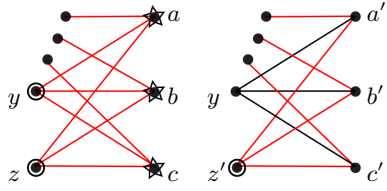
\begin{figure}[h!]
\begin{center}
\begin{overpic}[scale=0.7]{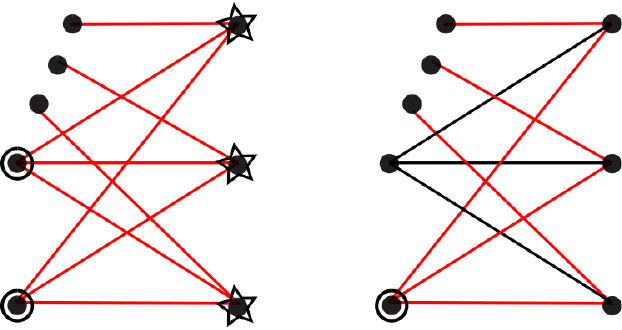}
\put(42, 48){\small $ a$}
\put(102, 48){\small $ a'$}
\put(-6, 23){\small $ y$}
\put(42, 23){\small $ b$}
\put(54, 23){\small $ y$}
\put(102, 23){\small $ b'$}
\put(-6, 0){\small $ z$}
\put(42, 0){\small $ c$}
\put(54, 0){\small $ z'$}
\put(102, 0){\small $ c'$}
\end{overpic}
\end{center}
\caption{
This illustrates the induced $\K$ subdivision in $D_x\G$, when $\L$ a bad edge between $x$ and $y$, and no other bad edges. }
\label{fig:isolated}
\end{figure}

Finally, we are left with the case that one of the sides has two edges (say there are edges between $x$ and $y$, and between $y$ and $z$) and the other side has zero or two bad edges. 
If the other side has zero bad edges, we have the graph on the left in Figure~\ref{fig:lambdas}.  
Otherwise assume without loss of generality that $[a,b]$ and $[b,c]$ are the two edges on the other side. 
 \begin{figure}[h!]
\begin{center}
 \begin{overpic}[scale=0.7]{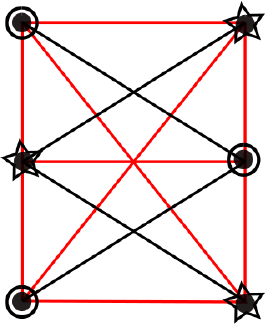}
\put(-10, 103){\small $ a$}
\put(85, 103){\small $ x$}
\put(-10, 50){\small $ b$}
\put(85, 50){\small $ y$}
\put(-10, 0){\small $ c$}
\put(85, 0){\small $ z$}
\end{overpic}
\end{center}
\caption{If one of 
$[x,b]$, $[z,b]$, $[a,y]$, $[c,y]$ is not an edge, then the red $\K$ subdivision is shorter than $\L$ and has at most $B(\L)$  bad edges.  
}
\label{fig:four-edges}
\end{figure}
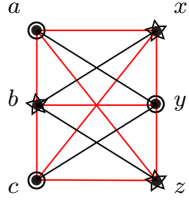
Now suppose that one of $[x,b]$, $[z,b]$, $[a,y]$, and $[c,y]$ is not an edge. Then we obtain a shorter $\K$ subdivision with at most $B(\L)$ bad edges, as shown in Figure~\ref{fig:four-edges}, establishing (ii) of the claim.  Otherwise, all of  $[x,b]$, $[z,b]$, $[a,y]$, and $[c,y]$ are edges, and we obtain the graph on the right in Figure~\ref{fig:lambdas}. 
\end{proof}

%%%%%%%%%%%%%%%%%%%%%%%%%%%%%%%%%%
\section{Boundaries of right-angled Coxeter groups defined by non-planar graphs.  } \label{sec:boundariesnonplanar} 
%%%%%%%%%%%%%%%%%%%%%%%%%%%%%%%%%%
In this section we prove Theorem~\ref{thm:main1}. That is, we show that if $\G$ is a non-planar triangle-free graph, then either $\bndry X_\G$ is non-planar for any $\CAT(0)$ space $X$ on which $W_\G$ acts geometrically, 
or $W_\G$ contains a finite index subgroup whose defining graph contains an induced copy of the graph in Figure~\ref{fig-badgraph}.  We will show below that the graph theoretic results of the previous section together with the doubling lemma (Lemma~\ref{lem:doubling})
can be used to reduce this to proving the following two propositions.
\begin{proposition}
\label{prop: non-planar 1}
If $\Delta$ is
 a $K_{3,3}$ subdivision and $W_\Delta$ acts geometrically on a $\CAT(0)$ space $X_\Delta$,
then $\bndry X_{\Delta}$ is non-planar.
\end{proposition}

\begin{proposition}
\label{prop: special graph}
Let $\Delta$
 be the graph on the right side in Figure~\ref{fig:lambdas}, subdivided enough so that it is triangle-free, and suppose that $W_\Delta$ acts geometrically on the $\CAT(0)$ space $X_\Delta$. 
 Then $\bndry X_{\Delta}$ is non-planar.
\end{proposition}

Before we prove the above propositions, we indicate how to use them to deduce 
Theorem~\ref{thm:main1}.

\begin{proof}[Proof of Theorem~\ref{thm:main1}]
Let $\G$ be a triangle-free non-planar graph.  Then Kuratowski's Theorem says that $\G$ contains a 
(not necessarily induced) $\K$ or $K_5$ subdivision.  
If $\G$ contains a $K_5$ subdivision, then 
Lemma~\ref{lem:K5} 
says that $\G_1$ contains a (not-necessarily induced) $\K$ subdivision, where $\G_1$ is either $\G$ itself, or the double of $\G$ over some vertex.  
Now we apply Proposition~\ref{prop:lambdas}
to $\G_1$.  We conclude that there is a graph $\G_2$ which is 
obtained from $\G_1$ by  a finite sequence of doubling moves, and an induced subgraph $\Delta$
of $\G_2$, such that $\Delta$ is either a $\K$ subdivision or one of the two graphs in Figure~\ref{fig:lambdas}.   Lemma~\ref{lem:doubling} 
implies that 
$W_{\G_2}$
is a finite-index subgroup of $W_\G$.

Now if $\Delta$ is either a $\K$ subdivision or the graph on the right in Figure~\ref{fig:lambdas}, then 
Propositions~\ref{prop: non-planar 1} and~\ref{prop: special graph} say that 
every boundary of $W_\Delta$  is non-planar.   The group $W_\Delta$ is a special subgroup of $W_{\G_2}$, and therefore every boundary of $W_{\G_2}$ is non-planar.  Now suppose that $W_\G$ acts geometrically on a $\CAT(0)$ space $X_\G$.  Then since $W_{\G_2}$ is a finite index subgroup of $W_\G$, $W_{\G_2}$ also acts geometrically on $X_\G$. Therefore the boundary of $X_\G$ is non-planar.

On the other hand, if $\Delta$ is the graph on the left in Figure~\ref{fig:lambdas} (which is the same as the graph in Figure~\ref{fig-badgraph}), then we have produced 
a finite-index subgroup of $W_\G$, namely $W_{\G_2}$, whose defining graph $\G_2$  contains an induced copy of the graph in Figure~\ref{fig-badgraph}.
\end{proof}

We now prove Propositions~\ref{prop: non-planar 1} and~\ref{prop: special graph}.  
We begin with a lemma about the boundaries of $\Theta$-graph subdivisions.  (By 
a \emph{$\Theta$-graph} we mean a graph with two essential vertices and three distinct edges between the essential vertices.)   Recall that a branch in a graph is an embedded path between essential vertices of the graph.  It contains its endpoints, but does not contain any other essential vertices.

\begin{lemma}
\label{lem: theta}
Let $\L$ be a $\Theta$-graph subdivision such that each branch has length at least 2, and at least one of the branches has length at least three. (See Figure~\ref{fig: theta}.) If $X_\L$ is a $\CAT(0)$ space on which $W_\L$ acts geometrically, then  $\bndry X_{\L}$ contains an embedded $\Theta$-graph.
\end{lemma}
We remark that the conclusion is true even without the condition that at least one branch has length at least three, but the above lemma is sufficient for our purposes.

\begin{figure}[h!]
\begin{overpic}[scale=1]{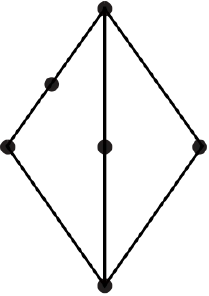}
\put(32, 102){\small $ a$}
\put(32, -9){\small $ b$}
\put(72, 47){\small $ z$}
\put(25, 47){\small $ y$}
\put(-9, 47){\small $ x$}
\put(11, 83 ){\scriptsize $\alpha_1$}
\put(37 ,68){\scriptsize $\alpha_2$}
\put(48, 83){\scriptsize $\alpha_3$}
\put(11, 17 ){\scriptsize $\beta_1$}
\put(37, 29){\scriptsize $\beta_2$}
\put( 50, 17){\scriptsize $\beta_3$}
\end{overpic}
\captionof{figure}{
A $\Theta$-graph subdivision with each branch of length at least two, and at least one branch of length at least 3. }
\label{fig: theta}
\end{figure}
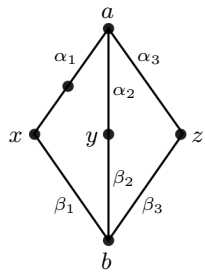

\begin{proof}
Assume that $\L$ has essential vertices $a$ and $b$, and non-essential vertices $x, y,$ and $z$ as shown in Figure~\ref{fig: theta}.  Let 
$\alpha_1, \alpha_2$, and $ \alpha_3$ be the paths from $a$ to $x$, $y$, and $z$, respectively. Define $\beta_1, \beta_2, \beta_3$ analogously with the $\beta_i$ incident to $b$. 
The hypotheses imply that at least one of these paths is subdivided, so we may assume $\alpha_1$ is subdivided.  

Let $\D_2$ and $\D_3$  be the cycles given by $(\beta_1, \alpha_1,\alpha_2,\beta_2)$
and 
$(\beta_1,\alpha_1,\alpha_3,\beta_3)$ respectively, 
and let $G_2$ and $G_3$ be the corresponding special subgroups. Consider the quasi-isometry (coming from the orbit map) between $X_\L$ and the Davis complex $\Sigma_\L$. Let $X_2$ denote the image of $\Sigma_{G_2}$ under this quasi-isometry based at the image of the identity vertex. Note that $X_2$ is quasi-convex.  Define $X_3$ analogously.  We we will use the cycles $\D_2$ and $\D_3$ to find a theta graph in $\bndry X_{\L}$.  

Now $\D_2$ and $\D_3$ are cycles of length at least $5$, so $G_2$ and $G_3$ are 
hyperbolic reflection groups acting  geometrically on $\bbh^2.$
Thus $\bndry X_2\cong \bndry X_3\cong S^1.$  
Since  $\D_2$ and $\D_3$  intersect in a path of length at least three, 
 the corresponding special subgroup is virtually free and its boundary, which is equal to $\bndry X_2\cap \bndry X_3$, is homeomorphic to either two points or a Cantor set, see Section~\ref{sec:RACG}.  
Furthermore, $\big\{(ab)^{\infty},(ba)^{\infty}\big\}\subset\bndry X_2\cap\bndry X_3$, and this set separates each of $\bndry X_2$ and $\bndry X_3$ into two components, as shown in Figure~\ref{fig: bndry intersection}.

 \begin{figure}[h!]
\begin{overpic}[scale=.75]{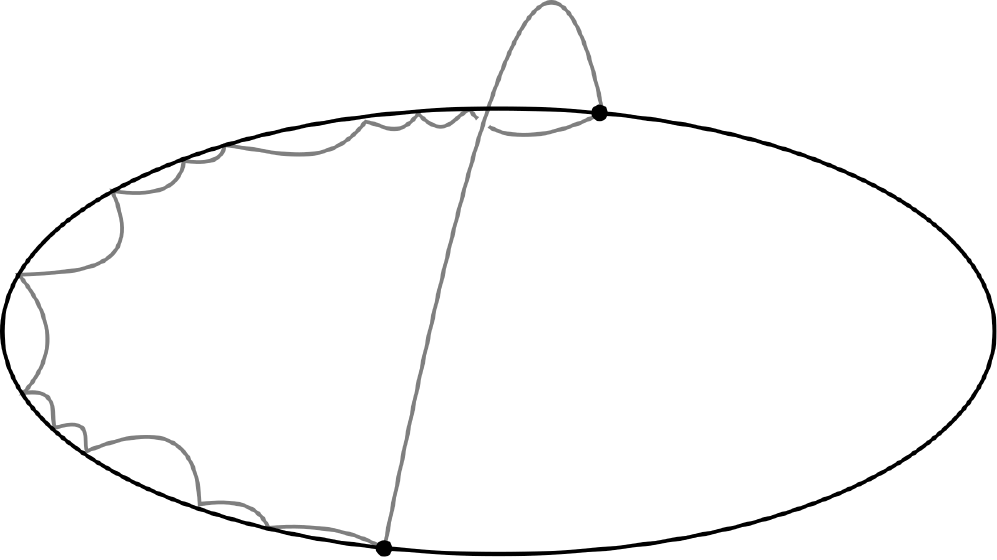}
\put(33, -4){\footnotesize $ (ba)^\infty$}
\put(61, 47){\footnotesize $ (ab)^\infty$}
\put(31, 5){\footnotesize $ \bndry X_2$}
\put(81, 2){\footnotesize $ \bndry X_3$}
\put(100.5, 20){\footnotesize $ Z$}
\end{overpic}
\medskip
\captionof{figure}{The figure shows $\bndry X_2$ in grey and $\bndry X_3$ in black. 
The intersection $\bndry X_2\cap\bndry X_3$ is a Cantor set.  The set $\big\{(ab)^{\infty},(ba)^{\infty}\big\}$ separates the grey and black circles into two components each.  The embedded $\Theta$-graph constructed in the proof consists of the grey circle together with the arc labeled $Z$. }
\label{fig: bndry intersection}
\end{figure}
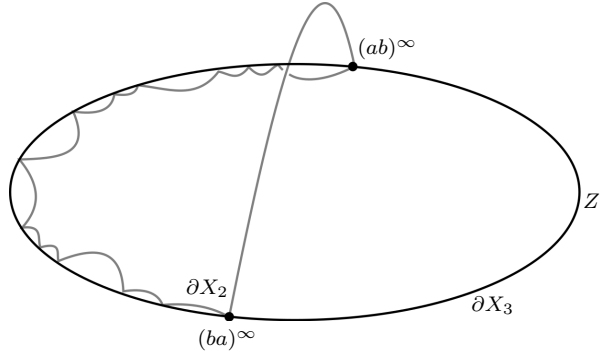

We now explicitly find a $\Theta$-graph in $\bndry X_\L$ with essential vertices $(ab)^{\infty}$ and $(ba)^{\infty}$.
 Let $Z$ be the subset of $\bndry X_3$ represented by Cayley graph geodesic rays whose first letter is the label of a vertex on the branch $\alpha_3 \cup \beta_3$.   
 Note that $Z$ is the arc of $\bndry X_3$ which intersects $\bndry X_2\cap\bndry X_3$ (and hence $\bndry X_2$) only at its endpoints $(ab)^{\infty}$ and $(ba)^{\infty}$.  
Then $\bndry X_3 \cup Z $ is an embedded $\Theta$-graph in $\bndry X_\L$. 
\end{proof}

We now prove Proposition~\ref{prop: non-planar 1}.  In  the proof we will use the 
embedded $\Theta$-graph in $\bndry X_\G$ that we constructed in the previous lemma.

\begin{proof}[Proof of Proposition~\ref{prop: non-planar 1}]
Let $\Delta$ be a $\K$ subdivision with essential vertex sets 
$\{a,b,c\}$ and $\{x,y,z\}$.
Let $\alpha_1, \alpha_2, \alpha_3$, be the branches from $a$ to $x$, $y$, and $z$, respectively. Define $\beta_1, \beta_2, \beta_3$ and $\gamma_1, \gamma_2, \gamma_3$ analogously with the $\beta_i$ incident to $b$, and the $\gamma_i$ incident to $c$. 
(See Figure~\ref{fig: K33}.)

If $\Delta$ is graph-isomorphic to $K_{3,3}$ (i.e.~if it is the trivial $\K$ subdivision) 
 then $\bndry X_{\Delta}$ is the join of two Cantor sets, and is therefore non-planar. 
Thus we may assume that at least one branch, say $\alpha_1$, of $\Delta$ is subdivided. 
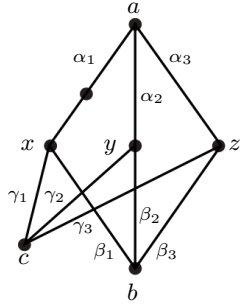
\begin{figure}[h!]
\begin{overpic}[scale=1.15]{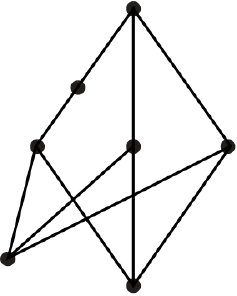}
\put(42, 102){\small $ a$}
\put(42, -9){\small $ b$}
\put(0, 5){\small $ c$}
\put(81, 49){\small $ z$}
\put(33, 49){\small $ y$}
\put(1, 49){\small $ x$}
\put(21, 83 ){\scriptsize $\alpha_1$}
\put(47 ,68){\scriptsize $\alpha_2$}
\put(58, 83){\scriptsize $\alpha_3$}
\put(29, 8 ){\scriptsize $\beta_1$}
\put(46.5, 22){\scriptsize $\beta_2$}
\put( 53, 8){\scriptsize $\beta_3$}
\put(-4,  30 ){\scriptsize $\gamma_1$}
\put(10, 30){\scriptsize $\gamma_2$}
\put( 21, 18){\scriptsize $\gamma_3$}
\end{overpic}
\captionof{figure}{A $K_{3,3}$ subdivision in which at least one branch is subdivided.}
\label{fig: K33}
\end{figure}

Observe that the union of the $\alpha_i$ and the $\beta_i$ is precisely the $\Theta$-graph subdivision 
$\L$ from Lemma~\ref{lem: theta}.  (See
Figure~\ref{fig: theta}.)
Retaining the notation of Lemma~\ref{lem: theta}, we see that $\bndry  X_2 \cup Z $ is 
an embedded $\Theta$-graph in $\bndry X_\Delta.$
To complete the proof we will find another half of a $\Theta$-graph which intersects 
$\bndry  X_2\cup Z $  in exactly three points. 
(See Figure~\ref{fig: bndry K33}.)

\begin{figure}[h!]
\begin{overpic}[scale=.65]{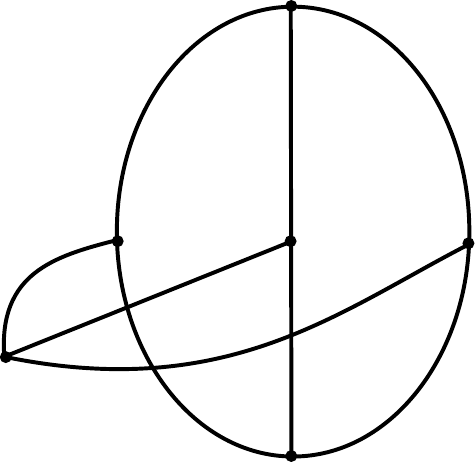}
\put(55, -8){\scriptsize $ (ba)^\infty$}
\put(55, 101){\scriptsize $ (ab)^\infty$}
\put(27, 46){\scriptsize $ (xy)^\infty$}
\put(63, 46){\scriptsize $ (yx)^\infty$}
\put(100.5, 46 ){\scriptsize $ (zx)^\infty$}
\put(-9,12 ){\scriptsize $ (ca)^\infty$}
\put(96, 70){\scriptsize $ Z$}
\put(44, 70){\scriptsize $\bndry X_2 $}
\put(11, 70){\scriptsize $\bndry X_2 $}
\put(3, 29){\scriptsize $ A$}
\put(77, 29){\scriptsize $ B$}
\end{overpic}
\captionof{figure}{An embedded $K_{3,3}$ in $\bndry X_{\Delta}$}
\label{fig: bndry K33}
\end{figure}

Let $\D'_2$ and $\D'_3$ be the cycles given by $(\gamma_1,
\alpha_1,\alpha_2,\gamma_2)$ and $(\gamma_1,\alpha_1,\alpha_3,\gamma_3)$, respectively. Define $ X'_2$ and $ X'_3$ to be the quasiconvex subspaces of $X_\Delta$ corresponding to the special subgroups generated by $\D'_2$ and $\D'_3$. Following the argument of Lemma~\ref{lem: theta} above with the cycles $\D_2$ and $\D'_2$, we see that $\bndry X_2$ and $\bndry X'_2$ intersect in a Cantor set, and that there is a closed arc $A$ of $\bndry X'_2$ such that $A\cap\bndry X_2=\big\{(xy)^{\infty},(yx)^{\infty}\big\}$. The interior of $A$  contains the point $(ca)^{\infty}$ and is disjoint from $\bndry X_2.$ 

Similarly considering the cycles  $\D_3$ and $\D'_3, $ we  conclude that 
$\bndry X_3 \cap \bndry X'_3$ is a Cantor set, and that there 
exists closed arc $C$ in $\bndry X'_3$ which intersects $\bndry X_3$ in $\big\{(xz^{\infty}),(zx)^{\infty}\big\}$ and whose interior contains $(ca)^{\infty}$.  Define $B$ to be the subarc of $C$ which connects $(ca)^{\infty}$ to $(zx)^{\infty}$. Notice that since $z$ is the label of the initial edge in the Cayley graph geodesic for the ray $(zx)^{\infty}$ we have that $(zx)^{\infty}\in Z$.
Moreover, points of $B$ are represented by Cayley graph geodesic rays  whose first letter is a label of a vertex on $\gamma_3$, see Lemma~\ref{lem:qc!},  and therefore $B$ intersects $\bndry  X_2 \cup Z \cup A$ exactly in the two points $(ca)^{\infty}$ and $(zx)^{\infty}$.  Thus $\bndry  X_2 \cup Z \cup A \cup C$ is an embedded 
$\K$ in  $\bndry X_{\Delta}$, 
as shown in Figure~\ref{fig: bndry K33}. 
\end{proof}

Next, we prove Proposition~\ref{prop: special graph} by showing that if $\Delta$ is the graph on the right in 
Figure~\ref{fig:lambdas}, 
 then $\bndry  X_\Delta$ is non-planar.   Figure~\ref{fig: badgraph1} below reproduces this graph, and also shows an alternate view of it. 
 Note that in our application (Theorem~\ref{thm:main1}), this graph is an induced subgraph of a triangle-free graph, which forces all of the black edges in the figure to be subdivided. 

\begin{figure}[h!]
\begin{overpic}[scale=0.75]
{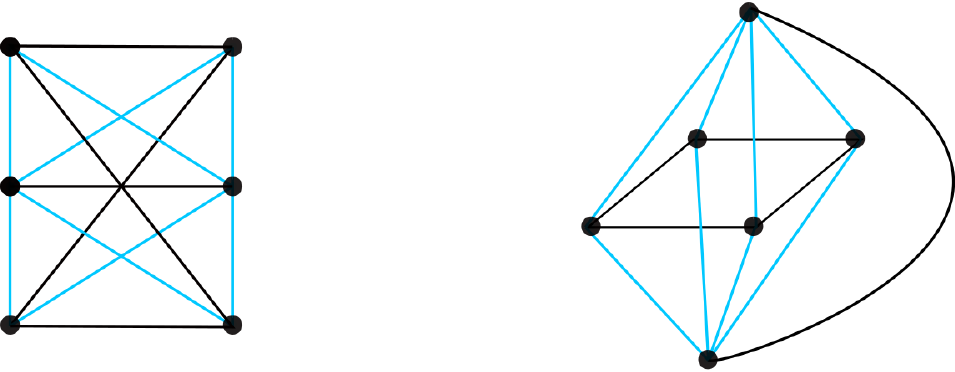}
\put(-4, 33){\small $a$}
\put(-4, 17){\small $ b$}
\put(-4, 3){\small $ c$}
\put(26, 3){\small $ z$}
\put(26, 17){\small $ y$}
\put(26, 33){\small $ x$}
\put(79.5, 13){\small $a$}
\put(78, 39){\small $ b$}
\put(74, 25){\small $ c$}
\put(91, 23){\small $ z$}
\put(75, -2){\small $ y$}
\put(57, 13){\small $ x$}
\end{overpic}
\captionof{figure}{
The figure shows two different views of the graph being considered in Proposition~\ref{prop: special graph}.
 The blue branches are edges. 
The black branches are necessarily subdivided.
}
\label{fig: badgraph1}
\end{figure}
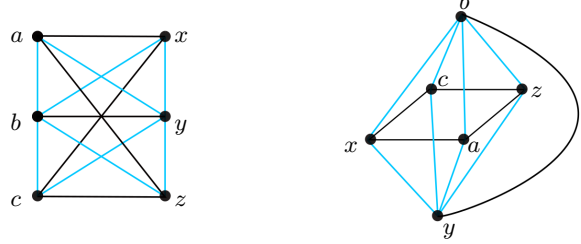

\begin{proof}[Proof of Proposition~\ref{prop: special graph}]
Let $x,y,z,  a,b,$ and $c$ be the essential vertices  of $\Delta$, as shown in Figure~\ref{fig: badgraph1}.  
We now define four cycles of $\Delta$ as follows:
 $\D_1= (x,c,z,a)$,
 $\D_2= (x,b,z,y)$, 
 $\D_3= (c,b,a, y)$, and  
$\D_4=(b, z, y)$. Note that the branch $[b, y]$ of $\D_4$ is subdivided by hypothesis.

\begin{figure}[h!]
\medskip
\begin{overpic}[scale=.45]{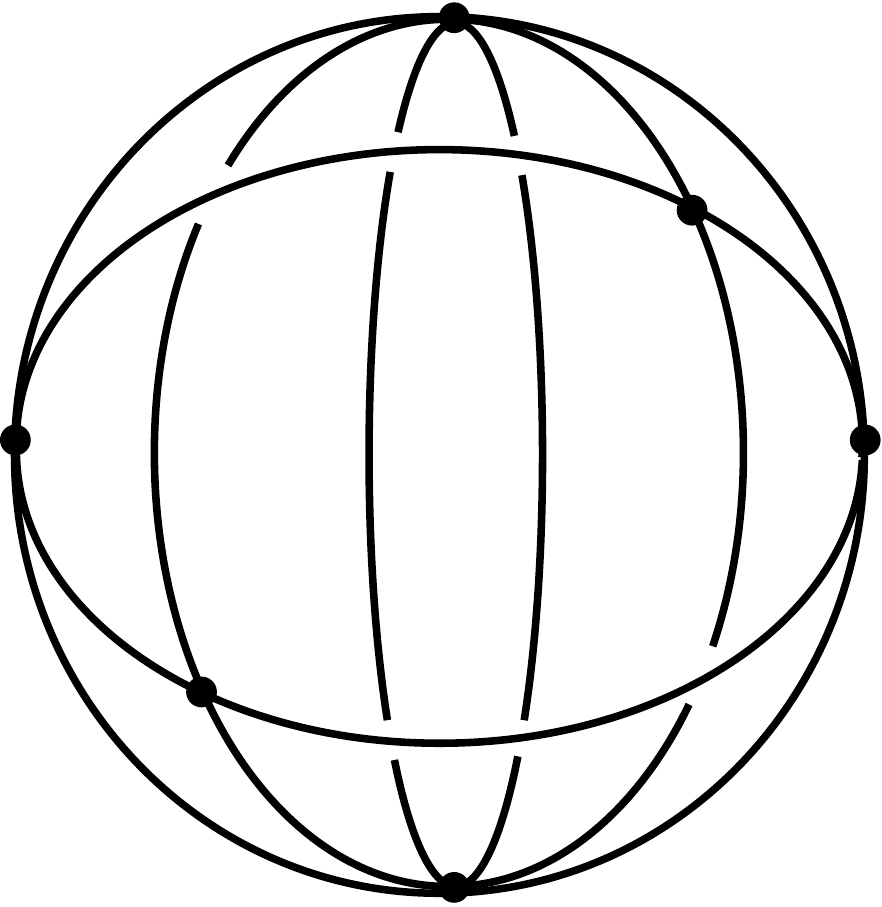}
\put(45, 103){\tiny $(by)^{\infty}$}
\put(60, 70){\tiny $(ca)^{\infty}$}
\put(98, 50){\tiny $ (zx)^{\infty}$}
\put(-18, 50){\tiny $ (xz)^{\infty}$}
\put(24, 26){\tiny $(ac)^{\infty}$}
\put(45, -5){\tiny $(yb)^{\infty}$}
\put(62, 24.5){\tiny $\bndry X_1$}
\put(7, 10){\tiny $\bndry X_2$}
\put(18, 50){\tiny $\bndry X_3$}
\put(42, 50){\tiny $\bndry X_4$}
\end{overpic}
\hspace{2cm}
\begin{overpic}[scale=.45]{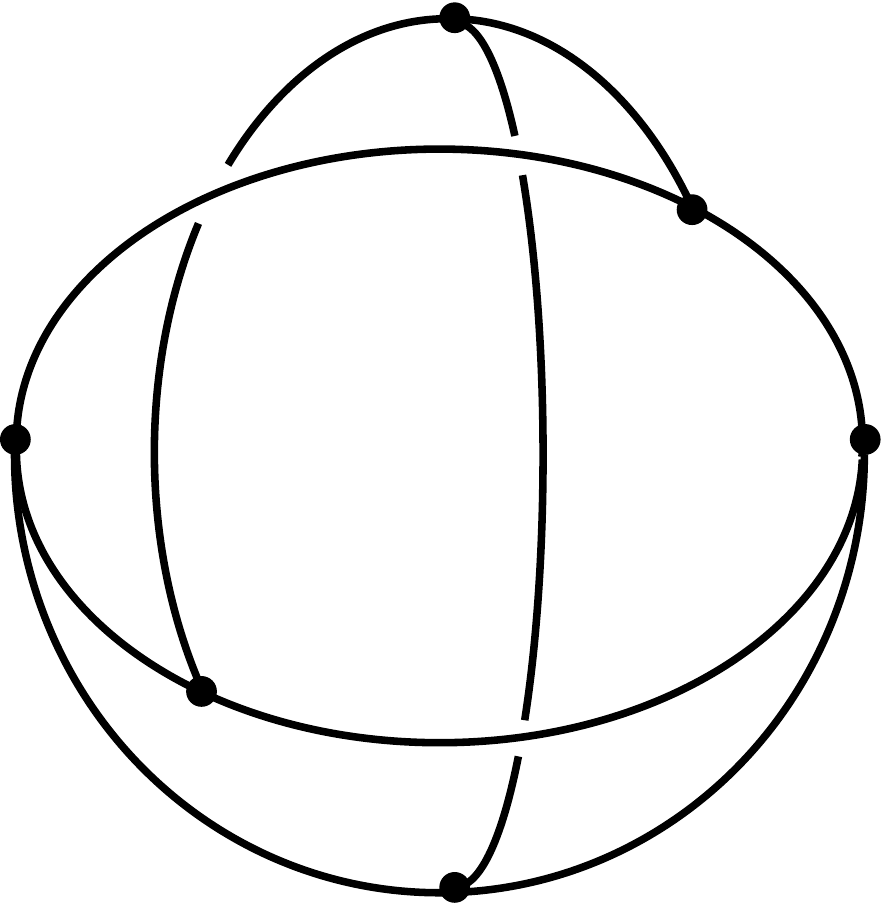}
\put(45, 103){\tiny $(by)^{\infty}$}
\put(61, 70){\tiny $(ca)^{\infty}$}
\put(98, 50){\tiny $ (zx)^{\infty}$}
\put(-18, 50){\tiny $ (xz)^{\infty}$}
\put(24, 26){\tiny $(ac)^{\infty}$}
\put(45, -5){\tiny $(yb)^{\infty}$}
\end{overpic}
\medskip
\captionof{figure}{
The picture on the left shows the pairwise intersections of the circles 
$\bndry X_1,\bndry X_2,$ $\bndry X_3,$ and $\bndry X_4$. The picture on the right shows an embedded $\K$ in 
$\bndry X_1 \cup\bndry X_2\cup\bndry X_3\cup \bndry X_4 \subset \bndry  X_\G$.}
\label{fig: bndrycircles}
\end{figure}

For $1 \le i \le 4$, let $ X_i$ be the quasi-isometrically embedded copy  of the Davis complex of the special subgroup generated by $\D_i$ which is based at the image identity vertex in $ X_\Delta$.
For each $i$, the cycle $\D_i$ has length at least $4$, so the corresponding special subgroup 
is virtually a surface group, and $\bndry X_i \cong S^1$.  The intersections of these circles are shown on the left in Figure~\ref{fig: bndrycircles}. 
 It follows that $\bndry  X_\G$ contains an embedded $\K$ as shown on the right in Figure~\ref{fig: bndrycircles}. 
\end{proof}

\section{The case of the bad graph} \label{sec:bad}

In this section we deal with the case that our defining graph contains an induced copy of the graph $\Pi$ (which is subdivided along some of the black edges).  As discussed in the introduction,  $W_\Pi$ has a planar boundary.   However, we will show in this section that the action of the group $W_\Pi$ on any $\CAT(0)$ boundary for $W_\Pi$ does not extend to the plane.  Therefore, when $\Pi$ is embedded in a graph $\G$ and $W_\G$ acts geometrically on a $\CAT(0)$ space $X$ such that $\partial X$ is connected, locally connected, and without local cut points, then $\partial X$ cannot be planar.

\begin{theorem} \label{thm:badisbad}
Let $\Pi$ be the graph in 
Figure~\ref{fig-badgraph}, with the non-blue edges subdivided so that $\Pi$ does not contain triangles.  Let $W_{\Pi}$ be the corresponding right-angled Coxeter group, and $\partial X_{\Pi}$ be the boundary of any proper $\CAT(0)$ space that $W_\Pi$ acts upon geometrically. Then:
\begin{enumerate}[(a)]
\item  $\bndry 
X_{\Pi}$ contains an embedded copy of the graph in Figure~\ref{fig:theta^theta}. 
\item The copies of the circles $A$, $B$, and $C$ shown in Figure~\ref{fig:theta^theta} are invariant under the induced action of $y$ on $\bndry X_{\Pi}$.
\item For any embedding of $\bndry X_{\Pi}$ in $S^2$, the induced action of $W_{\Pi}$ on $\bndry X_{\Pi}$ by homeomorphisms does not extend to $S^2$. 
\end{enumerate} 
 \end{theorem} 

The proof of Theorem~\ref{thm:badisbad} is delayed until the end of the section. First we discuss two corollaries. Recall that  a topological space is {\it planar} if it embeds in $S^2$. 

\begin{corollary} 
\label{cor:badnoplanar} 
Let $\G$ be a triangle-free finite simplicial graph which contains an induced copy of $\Pi$
(with non-blue edges possibly subdivided).
Suppose that $\partial X_{\Gamma}$ is connected, locally connected, and has no local cut points. Then $\partial X_{\Gamma}$ is not planar.
\end{corollary} 

\begin{proof}[Proof of Corollary~\ref{cor:badnoplanar} from Theorem~\ref{thm:badisbad}]
First, we claim that if $\partial X_{\G}$ satisfies all the hypotheses and is planar, then $\partial X_{\G}$ is a Sierpinski carpet.  Indeed, the Sierpinski carpet is the unique 1-dimensional topological space which is compact, connected, locally connected, planar, and has no cut points or local cut points~\cite{Whyburn}. Our assumption of triangle-free ensures that the boundary is $1$-dimensional by a theorem of Bestvina-Mess~\cite{Bes96, BM}. Since $\bndry X_{\G}$ is compact, and has no cut points~\cite[Theorem~1]{PS09}, this proves the claim. 

Now suppose that $\partial X_{\G}$ is a Sierpinski carpet $\mathcal{S}$ and $\Gamma$ contains an induced copy of the bad graph $\Pi$.  Then since $W_{\Pi}$ is a special subgroup of $W_{\G}$, the action of $W_{\Pi}$ on $\bndry 
X_{\Pi} \subset \partial X_\G$ extends to an action by homeomorphisms on the Sierpinski carpet. Every homeomorphism of the Sierpinski carpet preserves the set of non-separating circles.    Let $\mathcal{S} \cong S^2 \setminus \bigcup_iD_i$, where $\lbrace D_i\rbrace $ is a dense null family of open round discs  $D_i$ in $S^2$ such that $\bar D_i \cap \bar D_j= \emptyset$ if $i \neq j$.  Then the non-separating circles in $\mathcal{S}$ are exactly the boundaries of the $D_i$.   Thus every homeomorphism of $\mathcal{S}$ extends to $S^2$, so the action of $W_{\Pi}$ on $\partial W_{\G}$ extends to an action on $S^2$.  But this contradicts Theorem~\ref{thm:badisbad}, so $\partial X_{\G}$ must be non-planar. 
\end{proof} 
We can now put the pieces together to prove Theorem~\ref{thm:main2}, which we re-state for the convenience of the reader. 
\vskip .1 in 
\noindent {\bf Theorem~\ref{thm:main2}} \it{Let $\G$ be a triangle-free inseparable graph and let $X_\Gamma$ be a $\CAT(0)$ space on which $W_{\G}$ acts geometrically.  If $\G$ is non-planar and $\bndry X_\G$ is locally connected and contains no local cut points, then $\bndry X_\G$ is non-planar.  
 } 
\vskip .1 in

\begin{proof} Suppose that $\G$  has no triangles and is non-planar.  Then by Theorem~\ref{thm:main1} either $\partial X_\G$ is non-planar, or $W_\Gamma$ contains a finite index subgroup $W_\Lambda$ such that $\Lambda$ contains an induced copy of $\Pi$. Then since $\Gamma$ is inseparable, $\partial X_\Gamma$ ($ \cong X_\Lambda$) is connected. By hypothesis it is locally connected and has no local cut points. Therefore, by Corollary~\ref{cor:badnoplanar}, $\partial X_\G$ is non-planar. \end{proof} 

\begin{corollary} \label{cor:badmeansmenger} Suppose $\Gamma$ is triangle-free and inseparable such that $W_{\G}$ is either hyperbolic or $\CAT(0)$ with isolated flats. Suppose further that $\Gamma$ contains an induced (with non-blue edges possibly subdivided) copy of the bad graph $\Pi$ below in Figure~\ref{fig:badgraph}. Then $\partial X_{\G}$ is the Menger curve, where $X_{\G}$ is any proper $\CAT(0)$ that $W_\G$ acts on geometrically. 
\end{corollary}

\begin{proof}[Proof of Corollary~\ref{cor:badmeansmenger}]
By a theorem of Bestvina~\cite{Bes96}, the boundary $\partial X_{\G}$ is 1-dimensional since $\G$ is triangle-free.  As $W_{\G}$ is one-ended, $\partial X_{\G}$ is connected.  If $W_{\G}$ is hyperbolic, then the assumption that $\G$ is inseparable (so in particular $W_{\G}$ does not split over any two-ended group) implies that $\partial X_{\G}$ does not contain any local cut points by a result of Bowditch~\cite{Bow98}.  When $W_{\G}$ has isolated flats the analogous result is due to Haulmark~\cite{Hau18b}.  
In both cases,  $\partial X_{\G}$ is locally connected; this follows by Bestvina--Mess~\cite{BM} or Bowditch~\cite{Bow98} in the hyperbolic case
and by Hruska--Ruane~\cite{HR1} in the $\CAT(0)$ with isolated flats case.
 Therefore, all the hypotheses of Corollary~\ref{cor:badnoplanar} are satisfied, and the boundary is non-planar. Thus it must be a Menger curve~\cite{KK00, Hau18b}. 
\end{proof}

\begin{proof}[Proof of Theorem~\ref{thm:badisbad}]

Note that the boundary of $X_{{\Pi}}$ is possibly planar; the double over $y$ of $\Pi$ is a planar graph, see Lemma~\ref{lem:doubling}.  The point of the proof is that the action of $y$ on any boundary of $W_{{\Pi}}$ is non-planar.   Consider the graph $\Pi$  in Figure~\ref{fig:badgraph} below.  The black edges may be subdivided; the two edges connecting $x$ to $y$ and $y$ to $z$ are not. We will also assume that our defining graph has no triangles, which will force some of the edges to be subdivided. 

\begin{figure}[h!]
\medskip
\begin{overpic}[scale=0.7]{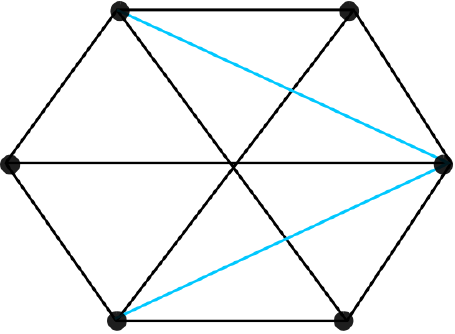}
\put(25,76){\small $x$}
\put(75,76){\small $a$}
\put(102,37){\small $y$}
\put(-6,37){\small $b$}
\put(24,-6){\small $z$}
\put(74,-6){\small $c$}
\end{overpic}
\caption{This is the bad graph $\Pi$.}
\label{fig:badgraph}
\end{figure}

We now define three cycles in $\Pi$ as follows: 
$A= (x, a, y)$,  $B = (x, b, y)$,
and $C= (x, c, y)$, where these denote the cycles defined by the branches traversing these essential vertices (and back to the initial essential vertex). 

  The associated special subgroups $W_{A}$, $W_{B}$, and $W_{C}$, intersect only in the finite subgroup generated by $x$ and $y$. Since every boundary of a $\CAT(0)$ space $X$ that a virtual surface group acts upon is $S^1$, we denote their boundaries using $\partial W_A$, etc. It follows that their boundaries $\partial W_A$, $\partial W_B$, and $\partial W_C$ 
are disjoint circles, each of which is invariant by the action of $y$.   By slight abuse of notation, we 
label these boundary circles in Figure~\ref{fig:theta^theta} below by $A$, $B$, and $C$ 
respectively.  

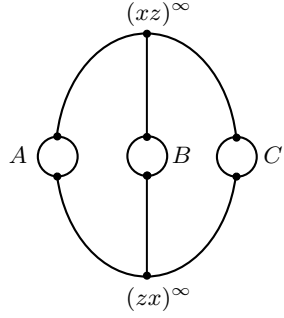
\begin{figure}[h!]
\medskip
\begin{overpic}
{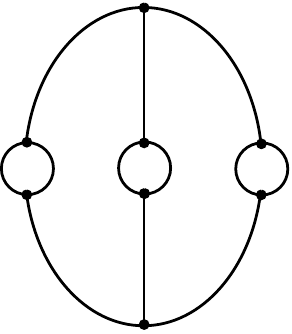}
\put(-11, 47){\small $A$}
\put(53, 47){\small $B$}
\put(89, 47){\small $C$}
\put(35, 103){\small $ (xz)^{\infty}$}
\put(35, -9){\small $ (zx)^{\infty}$}
\end{overpic}
\medskip
\caption{This figure is planar but the action of $y$ is non-planar}  
\label{fig:theta^theta}
\end{figure}

Furthermore,  the two fixed points of the action of $y$  in each circle are the endpoints of 
the loxodromic generated by the two vertices adjacent to $y$. 
For example, if the branch $[y,a]$  of $\Pi$ is 
not subdivided this is the limit set of the loxodromic element $xa$. Notice that the two points on the 
boundary $\bndry X_{\Pi}$ associated to the sequences $(xz)^\infty$ and and $(zx)^\infty$ are also 
fixed by the action of $y$, since both $x$ and $z$ commute with $y$. 

We wish to construct an 
embedded copy of Figure~\ref{fig:theta^theta} in $\bndry X_{{\Pi}}$. Throughout we will implicitly use  Lemma~\ref{lem:qc!}. Consider the following two 
cases:

\begin{itemize} 
\item $xa$ is a loxodromic.  In this case, use the cycle in $\Pi$ defined by $(x,a,z,b)$. 
 The induced special subgroup is virtually a surface group and so the boundary of this special subgroup is a circle.  This circle intersects the circle $A$ in the boundary of the subgroup defined by the branch $[x, a]$.  The endpoint associated to the sequence $(xa)^\infty$ can be connected to the endpoint of the sequence $(xz)^\infty$ using the arc in this circle which avoids the boundary of the subgroup associated to the branch $[x, a]$. Note that the endpoints of the loxodromic element $xz$ separate the endpoints of $xa$ from the endpoints of $xb$, when $xb$ is loxodromic.  This follows from the order of the elements around $A$. 

\item $xa$ is not loxodromic.  In this case $ya$ is loxodromic. In this case we will use the cycle in $\Pi$ defined by $(x,a,z,y)$.  Then we can connect the elements $(xz)^\infty$ and $(ya)^\infty$ by a an arc on this circle. 

\end{itemize}

We do this for each of $A$, $B$, and $C$ using the analogously defined circles.  If $xb$ is loxodromic, use the cycle in $\Pi$ defined by $(x,a,z,b)$, and using the other side from above, connect  $(xa)^\infty$  to the endpoint of the sequence $(xz)^\infty$. If $xb$ is not loxodromic, use $(x,b,z,y)$, connecting  $(xz)^\infty$ to the endpoints of $(yb)^\infty$.  Similarly, if $xc$ is loxodromic, we use the cycle in $\Pi$ defined by $(x,a,z,c)$.  The endpoints of $xc$ are separated from the endpoints of $xa$ by the endpoints of $xz$.  If $xc$ is not loxodromic, we use the cycle $(x,c,z,y)$, connecting $(xz)^\infty$ to $(yc)^\infty$. Moving by the action of $x$ takes the arcs connecting the endpoint of the ray $(xz)^\infty$ to arcs connecting the endpoints of the rays $(zx)^\infty$ and another point on the the circles $A$, $B$ and $C$. Therefore, in either case, we have a copy of the Figure~\ref{fig:theta^theta} in any $\CAT(0)$  boundary of $W_{{\Pi}}$, and hence in the boundary $\partial X_\G$,  since ${\Pi}$ is assumed to be an induced subgraph of $\Gamma$.  This proves (a) in the statement of 
Theorem~\ref{thm:badisbad}.

Now we claim that the action of the group element $y$ acting on Figure~\ref{fig:theta^theta} does not extend to the plane.  Up to relabelling, any embedding of Figure~\ref{fig:theta^theta} can be moved via a homeomorphism of the pair $(S^2, \text{Figure~\ref{fig:theta^theta}})$
 to the embedding given. 
The circles $A$, $B$ and $C$ are all invariant under the action of $y$, as are the boundary points $(xz)^\infty$ and $(zx)^\infty$.  The arcs connecting the points $(xz)^\infty$ and $(zx)^\infty$ to the circles $A$, $B$, and $C$ are not invariant, nor are the points where the arcs meet the circles unless $x$ and $a$ are both adjacent to $y$ in the graph ${\Pi}$, and the same for $B$ and $C$.    Assuming that $\partial X_{\G}$ is planar, $y$ sends Figure~\ref{fig:theta^theta} to a homeomorphic copy where the circles $A$, $B$ and $C$ have been flipped and the boundary points $(xz)^\infty$ and $(zx)^\infty$ are fixed.  Each of the arcs connecting the points to the circles go to arcs connecting the same points to the same circles. Consider the components of $S^2 \setminus \text{Figure}$~\ref{fig:theta^theta}. The component containing the point at infinity must go to the bounded component between $A$ and $B$, since the circle $A$ is flipped.  It must also go to the component between $B$ and $C$, since the circle $C$ is flipped.  This is a contradiction so the action does not extend to the plane.  This proves that the boundary $\partial X_{\G}$ is not planar. 
\end{proof}

\bibliographystyle{plain}
\bibliography{DHW}

\end{document}